\documentclass[11pt]{article}
\usepackage[margin=1in]{geometry}
\usepackage{amsmath,amssymb,amsthm,mathtools,mathrsfs}
\usepackage{enumitem}
\usepackage{algorithm}
\usepackage{algpseudocode}
\usepackage{booktabs}
\usepackage{hyperref}
\usepackage{xcolor}
\usepackage{microtype}
\usepackage{verbatim}
\usepackage{authblk}
\hypersetup{colorlinks=true, linkcolor=blue!50!black, citecolor=blue!50!black, urlcolor=blue!50!black}

\newtheorem{theorem}{Theorem}[section]
\newtheorem{proposition}[theorem]{Proposition}
\newtheorem{lemma}[theorem]{Lemma}
\newtheorem{corollary}[theorem]{Corollary}

\theoremstyle{remark}

\numberwithin{equation}{section}

\newcommand{\eps}{\varepsilon}

\newcommand{\cH}{\mathcal H}

\newcommand{\cR}{\mathcal R}

\newcommand{\diver}{\operatorname{div}}

\newcommand{\norm}[1]{\left\lVert #1\right\rVert}

\title{Mean field error estimate of the random batch method for vortex blob dynamics for the 2D Navier--Stokes Equation}
\author[1]{Zhenyu Huang\thanks{E-mail: zhenyuhuang@sjtu.edu.cn}}
  
\affil[1]{School of Mathematical Sciences, Institute of Natural Sciences, Shanghai Jiao Tong University, Shanghai, 200240, P.R.China}

\date{\today}

\begin{document}
\maketitle

\begin{abstract}
We propose and analyze the random batch vortex blob method for the 2D Navier--Stokes equation in vorticity form on the whole plane. The vortex blob method is based on an interacting particle system of $N$ particles with computational complexity of $O(N^2)$, which is reduced to $O(N)$ by the random batch method \cite{JinLiLiu2020}. Our main result is a quantitative law level mean field error estimate whose dependence on the blob radius remains algebraic. We treat the two main error mechanisms separately. The random batch error is controlled through a locally coupled auxiliary partition, a symmetric law comparison, and Fisher-information dissipation. The mean field fluctuation is estimated by exploiting the oddness and divergence-free structure of the Biot--Savart kernel. For smooth, strictly positive initial vorticity, we prove on every finite time interval a normalized relative-entropy bound of order
$
O\!\left(\varepsilon^{-4}\tau^2+N^{-1}
\right),
$
with constants independent of \(N\), \(\tau\), and \(\varepsilon\). Here $\tau$ is the batch refreshing interval and $\varepsilon$ is the blob radius. As a consequence, the fixed-particle marginals converge strongly in \(L^1\) to tensor products of the regularized vorticity solution when \(N\to\infty\) and \(\varepsilon^{-2}\tau\to0\).
\end{abstract}

\begingroup
\small


\medskip


\par
\endgroup

\section{Introduction}
We consider the stochastic $N$-particle point vortex system posed on the whole
Euclidean space $\mathbb{R}^2$, whose dynamics are governed by the system of
stochastic differential equations
\begin{equation}\label{eq:intro-pv}
  \mathrm{d}X^i(t)
  =
  \frac{1}{N-1}\sum_{j\neq i}K\bigl(X^i(t)-X^j(t)\bigr)\,\mathrm{d}t
  +
  \sqrt{2\sigma}\,\mathrm{d}B^i(t),
  \qquad 1\leq i\leq N.
\end{equation}
Here $X^i(t)\in\mathbb{R}^2$ denotes the position of the $i$-th vortex particle,
$\{B^i(t)\}_{i=1}^N$ are independent standard two-dimensional Brownian motions,
and $\sigma>0$ is the viscosity parameter, which determines the strength of the
diffusive noise. The interaction kernel $K$ is the two-dimensional Biot--Savart
kernel, given by
\begin{equation}\label{eq:intro-bs}
  K(x)
  =
  \frac{1}{2\pi}\frac{x^\perp}{|x|^2}
  =
  \frac{1}{2\pi}\frac{(-x_2,x_1)}{|x|^2},
  \qquad x=(x_1,x_2)\in\mathbb{R}^2\setminus\{0\}.
\end{equation}
In the present work, we focus on the simplest setting of indistinguishable
particles, in which all point vortices carry the same circulation. The extension
to vortices with general circulations will be addressed in future work.

The normalization factor $1/(N-1)$ in \eqref{eq:intro-pv} corresponds to the
mean field scaling. A central question is to understand the limiting behavior of
the empirical measure
\[
  \mu_t^N=\frac{1}{N}\sum_{i=1}^N\delta_{X^i(t)}
\]
as $N\to\infty$. Formally, under suitable assumptions, the empirical vorticity
is expected to converge to the solution of the two-dimensional incompressible
Navier--Stokes equation in vorticity form,
\begin{equation}\label{eq:intro-ns}
  \partial_t\omega+\operatorname{div}(u\omega)=\sigma\Delta\omega,
  \qquad u=K*\omega.
\end{equation}
The singularity of the Biot--Savart kernel at the origin, however, creates both
analytical and numerical difficulties. A classical way to overcome this issue is
the vortex blob method \cite{Chorin1973}, where the singular kernel is replaced by a mollified
one. More precisely, let $\varphi$ be a smooth mollifier with $\int \varphi =1$ and define
\[
  \varphi_\varepsilon(x)=\varepsilon^{-2}\varphi(x/\varepsilon),
  \qquad
  K_\varepsilon=K*\varphi_\varepsilon .
\]
The corresponding regularized particle system is then
\begin{equation}\label{eq:intro-blob}
  \mathrm{d}X^i_\varepsilon(t)
  =
  \frac{1}{N-1}\sum_{j\neq i}
  K_\varepsilon\bigl(X^i_\varepsilon(t)-X^j_\varepsilon(t)\bigr)\,\mathrm{d}t
  +
  \sqrt{2\sigma}\,\mathrm{d}B^i(t),
  \qquad 1\leq i\leq N.
\end{equation}
The vortex blob approximation has been extensively studied in the numerical
analysis of fluid equations and, more recently, in particle methods for
aggregation and diffusion equations \cite{carrillo2019blob,craig2016blob}. 

Even for the regularized system \eqref{eq:intro-blob}, a direct time
discretization requires the computation of all pairwise interactions at each
time step, leading to an $\mathcal{O}(N^2)$ computational cost per step. To
reduce this cost, one may employ the random batch method proposed by \cite{JinLiLiu2020}. At each time step, the
$N$ particles are randomly divided into $n$ batches of size $p\ll N$, and each
particle interacts only with the particles in the same batch. This reduces the
computational complexity from $\mathcal{O}(N^2)$ to $\mathcal{O}(pN)$ per time
step, while preserving the mean field structure in an unbiased stochastic
sense. The RBM algorithm for \eqref{eq:intro-blob} is given in Algorithm \ref{alg:rbm-vortex-blob}.
\begin{algorithm}[H]
\caption{Random batch vortex blob method in $\mathbb{R}^2$}
\label{alg:rbm-vortex-blob}
\begin{algorithmic}[1]
\State Choose $N$, batch size $p\mid N$, blob radius $\eps$, time step $\tau$.
\For{$k=1,\ldots,\lfloor T/\tau\rfloor-1$}
  \State Divide $\{1,2,\cdots,N\}$ into $N/p$ batches randomly.
    \State Update $\tilde{X}_\eps^{i} (i=1,\cdots,N)$ by solving the following SDE with $t\in[k\tau,(k+1)\tau)$:
    \begin{equation}\label{intro-rbm-blob-particle}
      \mathrm{d} \tilde{X}_\eps^{i}=
      \frac{1}{p-1}
      \sum_{\substack{j\in\xi_k(i),  j\ne i}}
      K_\eps(\tilde{X}_\eps^{i}-\tilde{X}_\eps^{j}) \mathrm{d} t
      +
      \sqrt{2\sigma}\,\mathrm{d} B^i.
  \end{equation}
  \EndFor
\end{algorithmic}
\end{algorithm}
The main purpose of this paper is to prove that whether the empirical
distribution generated by the random batch dynamics still converges to the vorticity solution of the two-dimensional
Navier--Stokes equation \eqref{eq:intro-ns} with regularized Biot-Savart kernel  in the
mean field limit. More precisely, we establish quantitative
propagation-of-chaos type error estimates for the random batch approximation of
the regularized stochastic point vortex system with fixed blob parameter $\varepsilon>0$
\begin{equation}\label{eq:intro-reg-ns}
  \partial_t\omega_\varepsilon
  +
  \operatorname{div}\bigl((K_\varepsilon*\omega_\varepsilon)
  \omega_\varepsilon\bigr)
  =
  \sigma\Delta\omega_\varepsilon,
  \qquad
  K_\varepsilon=K*\varphi_\varepsilon .
\end{equation}
Our analysis is based on the
relative entropy method, which allows us to quantify the combined effects of the
mean field limit, the kernel regularization and the additional stochastic error
introduced by the RBM mechanism. Let $\tilde{F}^N_{\varepsilon}(t)$ denote the joint law of the $N$ particles
evolved under the random batch dynamics \eqref{intro-rbm-blob-particle}
\[
\tilde{F}^N_{\varepsilon}(t) = \text{Law} (\tilde{X}_{\eps}^1(t),\cdots, \tilde{X}_{\eps}^N(t)),
\]
and let
$\omega_\varepsilon^{\otimes N}$ be the tensorized law associated with the
solution of \eqref{eq:intro-reg-ns}. We prove estimates on the normalized
relative entropy
\begin{equation}\label{relative-entropy-rescaled}
  \mathcal H_N\bigl(
    \tilde{F}^N_{\varepsilon}
    \,\big|\,
    \omega_\varepsilon^{\otimes N}
  \bigr)
  :=
  \frac1N
  \int_{\mathbb R^{2N}}
  \tilde{F}^N_{\varepsilon}
  \log
  \frac{
    \tilde{F}^N_{\varepsilon}
  }{
    \omega_\varepsilon^{\otimes N}
  }
  \,\mathrm dX^N ,
\end{equation}
where $X^N=(x_1,\ldots,x_N)$. The resulting bounds quantify the error in terms
of the particle number $N$, the time step of the RBM $\tau$ and the regularization scale $\varepsilon$. 
\subsection{Main results}
We first clarify our settings and notations. Recall the dynamics of the RBM in Algorithm \ref{alg:rbm-vortex-blob}. For each time step $t_k=k\tau$, we denote by $\xi_k$ a random partition of $\{1, \cdots, N\}$ and let $\boldsymbol{\xi}:=\left(\xi_1, 
\xi_2, \cdots\right)$ represent the sequence of batch partitions. We also denote by $\xi_k(\cdot)$ the unique batch in the partition $\xi_k$ such that
\[
\xi_k(i):=\left\{i, i_1, \cdots, i_{p-1}\right\}.
\] It is straightforward to show that
\[
j\in\xi_k(i)
\quad\Longleftrightarrow\quad
i\in\xi_k(j).
\]
We take $\Omega$ as the sample space equipped with the uniform probability measure $\mathbb{P}$, and define the filtration $\left\{\mathcal{F}_k\right\}_{k \geq 0}$, where $\mathcal{F}_k$ is the $\sigma$-algebra generated by $\left\{\xi_j, j \leq k\right\}$ and the Brownian motions before time $t_k$ in the particle system. In what follows, we will use the symbol $\mathbb{E}$ to indicate expectation over this probability space. 

For a fixed batch division sequence $\boldsymbol{\xi}$, we denote the joint law of the system at time $t$
\[
\tilde{F}^{N,\boldsymbol{\xi}}_\varepsilon(t) = \text{Law} (\tilde{X}^{1,\boldsymbol{\xi}}_\eps(t),\cdots, \tilde{X}^{N,\boldsymbol{\xi}}_\eps(t)).
\]
Applying the It\^o formula to the random batch particle system \eqref{intro-rbm-blob-particle} and noting that $\operatorname{div} K_\eps = 0$, one can conclude that the joint law $\tilde{F}^{N,\boldsymbol{\xi}}_\varepsilon$ solves the
following Liouville equation for $t\in[k\tau,(k+1)\tau)$:
\begin{equation}\label{intro-rbm-blob-liouville}
\partial_t \tilde{F}^{N,\boldsymbol{\xi}}_\varepsilon+ \sum_{i=1}^N \left(\frac{1}{p-1} \sum_{\substack{j\in\xi_k(i),  j\ne i}}K_\varepsilon\left(x_i-x_j\right)\right) \cdot \nabla_{x_i} \tilde{F}^{N,\boldsymbol{\xi}}_\varepsilon=\sigma \sum_{i=1}^N \Delta_{x_i} \tilde{F}^{N,\boldsymbol{\xi}}_\varepsilon.
\end{equation}
By averaging over all possible realizations of the random partitions, one can obtain the joint law of the system described by \eqref{intro-rbm-blob-particle}
\[
\tilde{F}^{N}_\varepsilon=\mathbb{E}_{\boldsymbol{\xi}} \tilde{F}^{N,\boldsymbol{\xi}}_\varepsilon.
\]
where the expectation $\mathbb{E}_{\boldsymbol{\xi}}$ is taken over all possible random batch divisions $\boldsymbol{\xi}$.

As for the regularized particle system \eqref{eq:intro-blob}, it is well-known in the mean field theory that the limit self-consistent McKean-Vlasov system writes
\[
\mathrm{d} X_\eps(t)=\left(K_\eps * \omega_\eps\right)\left(X_\eps(t)\right)+\sqrt{2 \sigma} \mathrm{~d} B(t),
\]
where $\omega_\eps$ denotes the law of $X_\eps$. By applying the Itô formula, one also obtains that $\omega_\eps$ solves the vorticity formulation of the regularized 2D Navier–Stokes equation 
\[
  \partial_t\omega_\varepsilon
  +
  \operatorname{div}\bigl((K_\varepsilon*\omega_\varepsilon)
  \omega_\varepsilon\bigr)
  =
  \sigma\Delta\omega_\varepsilon.
\]
After direct and simple computations, one can verify that the $N$-tensorized law of $\omega_\eps$ denoted by $\omega_\varepsilon^{\otimes N}$ can solve the following PDE:
\begin{equation}\label{intro-tensorized-pde}
    \partial_t \omega_\varepsilon^{\otimes N}
    +\sum_{i=1}^N
    \left(K_\varepsilon*\omega_\varepsilon\right)(x_i)
    \cdot\nabla_{x_i}\omega_\varepsilon^{\otimes N}
    =\sigma\sum_{i=1}^N\Delta_{x_i}\omega_\varepsilon^{\otimes N}.
\end{equation}
Our main result is the explicit estimate for the normalized relative entropy between $\tilde{F}^{N}_\varepsilon$ and $\omega_\varepsilon^{\otimes N}$ defined in \eqref{relative-entropy-rescaled}.
\begin{theorem}[Finite-time relative-entropy estimate]
\label{thm:main}
Let $\sigma>0$, $T>0$, $N\geq2$, and let $2\leq p\leq N$ satisfy
$p\mid N$. Let $\varphi\in C_c^\infty(\mathbb R^2)$ be nonnegative and even,
with
$
  \int_{\mathbb R^2}\varphi(x)\,\mathrm dx=1,
$
and define
\[
  \varphi_\varepsilon(x)=\varepsilon^{-2}\varphi(x/\varepsilon),
  \qquad
  K_\varepsilon=K*\varphi_\varepsilon,
  \qquad 0<\varepsilon\leq1.
\]
Assume that $\omega_0$ is a strictly positive probability density and that
there exist constants $a_0,C_0,C_1,C_2>0$ such that
\begin{equation}
\label{eq:main-initial-assumptions}
\begin{aligned}
  \omega_0(x)
  &\leq C_0e^{-a_0|x|^2},
  \\
  |\nabla\log\omega_0(x)|^2
  &\leq C_1(1+|x|^2),
  \\
  |\nabla^2\log\omega_0(x)|
  &\leq C_2(1+|x|^2).
\end{aligned}
\end{equation}
Let $\omega_\varepsilon$ be the solution of
\eqref{eq:intro-reg-ns} with initial data $\omega_0$, and let $\tilde F_\varepsilon^N(t)$ be the joint
law of the random batch dynamics in Algorithm~\ref{alg:rbm-vortex-blob} with symmetric initial density $\tilde F_\varepsilon^N(0)$.
Set
\[
  h_N^0
  :=
  \mathcal H_N\!\left(
    \tilde F_\varepsilon^N(0)
    \,\middle|\,
    \omega_0^{\otimes N}
  \right).
\]
Then, for every $\tau>0$, there exists a constant $C_T>0$, depending only
on $T$, $\sigma$, $p$, $\varphi$, and the constants in
\eqref{eq:main-initial-assumptions}, but independent of
$N$, $\tau$, and $\varepsilon$, such that
\begin{equation}
\label{main-result}
  \sup_{0\leq t\leq T}\mathcal H_N\!\left(
    \tilde F_\varepsilon^N(t)
    \,\middle|\,
    \omega_\varepsilon(t)^{\otimes N}
  \right)
  \leq
  C_T\left[
    h_N^0
    +
    \varepsilon^{-4}\tau^2\left(1+  h_N^0\right)
    +
    \frac1N
  \right].
\end{equation}
In particular, for the product initial law
$\tilde F_\varepsilon^N(0)=\omega_0^{\otimes N}$,
\begin{equation}
\label{main-result-product}
  \sup_{0\leq t\leq T}
  \mathcal H_N\!\left(
    \tilde F_\varepsilon^N(t)
    \,\middle|\,
    \omega_\varepsilon(t)^{\otimes N}
  \right)
  \leq
  C_T\left(
    \varepsilon^{-4}\tau^2+\frac1N
  \right).
\end{equation}
\end{theorem}
It is well-known that one can bound the $L^1$ norm with square root of the relative entropy by the Csisz{\'a}r
Kullback–Pinsker inequality, see for instance \cite{bolley2005weighted}.

\begin{corollary}[$L^1$-propagation of chaos]
\label{cor:L1-propagation-chaos}
Under the assumptions of Theorem~\ref{thm:main}, let
$\tilde F_{\varepsilon,k}^N(t)$ denote the $k$-particle marginal of
$\tilde F_\varepsilon^N(t)$, where $k\geq1$ is fixed. Then
\begin{align}
\label{eq:L1-marginal-bound}
 &\sup_{0\leq t\leq T}
 \left\|
   \tilde F_{\varepsilon,k}^N(t)
   -
   \omega_\varepsilon(t)^{\otimes k}
 \right\|_{L^1(\mathbb R^{2k})}
 \notag\\
 &\qquad\leq
 C_{T,k}\left[
   (h_N^0)^{1/2}
   +
   \varepsilon^{-2}\tau(1+h_N^0)^{1/2}
   +
   N^{-1/2}
 \right],
\end{align}
where $C_{T,k}$ is independent of $N$, $\tau$, and $\varepsilon$.
Consequently, for fixed $\varepsilon>0$, if
$
  h_N^0\longrightarrow0, \tau\longrightarrow0, N\longrightarrow\infty,
$
then for every fixed $k$,
\[
  \left\|
    \tilde F_{\varepsilon,k}^N
    -
    \omega_\varepsilon^{\otimes k}
  \right\|_{L^\infty(0,T;L^1(\mathbb R^{2k}))}
  \longrightarrow0.
\]
If
$\tilde F_\varepsilon^N(0)=\omega_0^{\otimes N}$, then
\begin{equation}
\label{eq:L1-product-rate}
  \left\|
    \tilde F_{\varepsilon,k}^N
    -
    \omega_\varepsilon^{\otimes k}
  \right\|_{L^\infty(0,T;L^1(\mathbb R^{2k}))}
  \leq
  C_{T,k}\left(
    \varepsilon^{-2}\tau+N^{-1/2}
  \right).
\end{equation}
\end{corollary}

The significance of Theorem \ref{thm:main} lies in the fact that its dependence on the blob radius is only {\it algebraic}. Indeed, if one treats $K_\varepsilon$ merely as a generic smooth Lipschitz kernel and follows the standard argument in \cite{HuangJinLi2025}, the natural stability coefficient would involve
\[
\|\nabla K_\varepsilon\|_{L^\infty}
\lesssim \varepsilon^{-2},
\]
which would lead to an error bound of the form $\exp(C_T\varepsilon^{-2})$. Such an exponential dependence on the regularization scale is highly unfavorable, both analytically and computationally. In particular, in order to reach a prescribed accuracy, one would have to take the time step exponentially small in $\varepsilon^{-1}$, which is impractical in numerical applications. A key point of the present work is therefore to exploit the special structure of the mollified Biot–Savart kernel and to refine the stability analysis so that all constants depend on $\varepsilon$ only in an algebraic way. This is one of the main distinctions between the present paper and the general Lipschitz-kernel framework developed in \cite{HuangJinLi2025}.

The two-dimensional Biot--Savart kernel with even mollifier has more structure than a generic Lipschitz force. On $\mathbb{R}^2$, one has
\begin{equation}\label{eq:intro-structure}
  \diver K_\varepsilon=0,
  \qquad K_\varepsilon(x)=-K_\varepsilon(-x),
  \qquad |z||K_\varepsilon(z)|\lesssim1.
\end{equation}
These identities are responsible for the cancellation in the relative entropy method for singular vortex models.  In the present work they allow us to estimate the full empirical Biot--Savart fluctuation without putting $\norm{\nabla K_\eps}_{L^\infty}$ into the Gronwall exponent.

The latter is recovered by taking the subsequent vanishing-blob limit
$\varepsilon\to0$. This limit is a standard consequence of the compactness
argument underlying Leray's mollified approximation of the Navier--Stokes
equations. Indeed, since
\[
  K_\varepsilon*\omega_\varepsilon
  =
  (K*\varphi_\varepsilon)*\omega_\varepsilon
  =
  \varphi_\varepsilon*(K*\omega_\varepsilon),
\]
the regularized vorticity equation \eqref{eq:intro-reg-ns} can be viewed as the
curl formulation of the Leray-regularized Navier--Stokes equations, in which the
transport velocity is mollified. Therefore, under standard assumptions on the
initial vorticity, every limit point of $\omega_\varepsilon$ as
$\varepsilon\to0$ is a weak solution of the two-dimensional Navier--Stokes
equations in vorticity form; by uniqueness in the corresponding two-dimensional
vorticity class, the whole family converges to the Navier--Stokes solution. We
refer to the classical work of Leray \cite{Leray1934} and to the modern
presentations in \cite{GuermondOdenPrudhomme2004,BerselliSpirito2021} for this
mollified approximation and compactness argument. Combining the subsequent vanishing-blob limit with the \(L^1\)-propagation of chaos \eqref{eq:L1-product-rate}
established above gives the complete convergence of the random batch vortex
blob approximation.
\begin{corollary}[The RBM particle approximation of 2D Navier--Stokes equation]
\label{cor:sequential-particle-batch-blob-limit}
Let \(\omega\) be the solution of the two-dimensional Navier--Stokes
vorticity equation \eqref{eq:intro-ns} with initial data \(\omega_0\). 
Let \(\tilde F_{\varepsilon,k}^{N,\tau}\) denote the explicit
dependence of the \(k\)-particle marginal on the batch time step
\(\tau\). Assume that
$
  h_N^0
  \longrightarrow0.
$
Then, for every fixed \(k\geq1\) and \(T>0\),
\begin{equation}
\label{eq:sequential-particle-batch-blob-limit}
  \lim_{\varepsilon\downarrow0}
  \lim_{\tau\downarrow0}
  \limsup_{N\to\infty}
  \sup_{0\leq t\leq T}
  \left\|
    \tilde F_{\varepsilon,k}^{N,\tau}(t)
    -
    \omega(t)^{\otimes k}
  \right\|_{L^1(\mathbb R^{2k})}
  =
  0.
\end{equation}
\end{corollary}
\subsection{Related work}
Point vortices enter the analysis of two-dimensional incompressible flow in two ways which are closely related but conceptually different. In the numerical-analysis
literature, a vortex method is a Lagrangian discretization of the vorticity
formulation of the Euler or Navier--Stokes equations, and the principal issue is
to control the combined effects of spatial sampling, core regularization, and
time discretization.  In the probabilistic mean field literature, by contrast,
the stochastic point-vortex system is itself an exchangeable interacting
particle system, and the main question is propagation of chaos as the number of
vortices tends to infinity. 

The random vortex method goes back at least to Chorin's construction for
slightly viscous flow~\cite{Chorin1973}.  For smooth Euler solutions, Hald
established an early convergence theory for vortex methods~\cite{Hald1979}, and
Beale-Majda developed a systematic analysis of high-order accurate vortex
schemes in two and three dimensions~\cite{BealeMajda1982I,BealeMajda1982II}.
Goodman-Hou-Lowengrub proved convergence of the point-vortex method for
the two-dimensional Euler equations~\cite{GoodmanHouLowengrub1990}, while the
corresponding three-dimensional analysis was carried out by Hou-Lowengrub~\cite{HouLowengrub1990}.  For viscous flow, Goodman proved convergence
of the random-vortex method~\cite{Goodman1987}, and Long subsequently obtained
an almost optimal probabilistic convergence estimate in two
dimensions~\cite{Long1988}.  These smooth-solution theories were later extended
to substantially rougher data. Liu-Xin proved convergence of vortex methods
to weak two-dimensional Euler solutions whose initial vorticity is a
one-signed finite Radon measure with locally finite kinetic energy, a class that
includes vortex-sheet data, under an appropriate relation between the particle
spacing and the blob scale~\cite{LiuXin1995}.  

A parallel probabilistic line of work studies stochastic point vortices as a
microscopic approximation of the viscous vorticity equation.  Early qualitative
mean field results were obtained by Marchioro-Pulvirenti~\cite{MarchioroPulvirenti1982}.  Osada proved
propagation of chaos for bounded initial densities when the viscosity is
sufficiently large~\cite{Osada1986,Osada1987}. A major advance was made by Fournier-Hauray-Mischler~\cite{FournierHaurayMischler2014}. By using entropy and Fisher-information
compactness, they treated the full plane $\mathbb R^2$, arbitrary positive
viscosity, and general circulations, for initial laws with finite entropy and a
positive moment. Another breakthrough was the quantitative propagation of chaos through the relative-entropy method of
Jabin-Wang~\cite{JabinWang2018}.  Their argument combines an entropy
evolution identity with a sharp exponential law of large numbers and applies
to a broad class of rough interactions, including kernels in
$\dot W^{-1,\infty}$.  Guillin-Le Bris-Monmarch\'e subsequently obtained uniform-in-time propagation of chaos for the
two-dimensional vortex model on the torus and for related singular stochastic
systems~\cite{GuillinLeBrisMonmarche2025}.  For the whole-space, Feng-Wang proved the first quantitative relative-entropy estimate for the viscous vortex
model by using Li--Yau and Hamilton type estimates for the
mean field vorticity equation \cite{FengWang2026}.  They also extended the
analysis to general circulations in a high viscosity regime \cite{FengWang2026Nonlinearity}. More general accounts of mean field limits and their
probabilistic and kinetic formulations can be found in the surveys \cite{Golse2016, Jabin2014, JabinWang2017}.

The Random Batch Method (RBM) was introduced by Jin-Li-Liu as a generic
way to accelerate large interacting particle systems~\cite{JinLiLiu2020}. Their original analysis emphasized that the time step and the error bound can be made independent of the total particle number under suitable assumptions. The
subsequent numerical analysis established strong and weak convergence estimates
for systems with disparate species and particle weights~\cite{JinLiLiu2021},
and treated second-order interacting systems~\cite{JinLiSun2022}.


The RBM has already been used in a
wide range of problems.  In electrostatic and molecular simulations, examples
include particle methods for the Poisson--Nernst--Planck and
Poisson--Boltzmann equations~\cite{LiLiuTang2022}, Monte Carlo treatment of
many-body systems with singular kernels~\cite{LiXuZhao2020}, the random batch
Ewald method for Coulomb interactions~\cite{JinLiXuZhao2021}, and its
superscalable implementation~\cite{LiangEtAl2022}.  In optimization and
sampling, the RBM has been incorporated into consensus-based global
optimization~\cite{CarrilloJinLiZhu2021} and into a stochastic version of Stein
variational gradient descent~\cite{LiLiLiuLiuLu2020}.  Quantum applications
include random batch quantum Monte Carlo~\cite{JinLiQuantum2020} and the RBM for
$N$-body quantum dynamics~\cite{GolseJinPaul2021}.  For kinetic equations,
Carrillo-Jin-Tang developed the RBM particle methods for the
homogeneous Landau equation~\cite{CarrilloJinTang2022}.  A particularly
interesting recent development uses batches of size two as genuine collision
pairs by Du-Li \cite{DuLi2024}. They introduced a collision-oriented interacting particle system
and analyzed molecular chaos for regularized Landau-type equations, while \cite{DuLiXieYu2025} designed a structure-preserving
collisional particle method for the Landau equation.
These examples show that the RBM can also encode a
physically meaningful sparse interaction mechanism in suitable kinetic settings.

The mean field analysis of the RBM has two distinct formulations.  Jin-Li first
studied an iterated limit in which $N\to\infty$ at a fixed batch time step
$\tau$, obtaining a discrete-time nonlinear mean field dynamics, and then
proved that this dynamics converges in Wasserstein distance to the usual
nonlinear Fokker--Planck equation as $\tau\to0$~\cite{JinLi2022}.  This limit is
structurally different from classical propagation of chaos because the random
repartitioning imposes a fresh product structure at the switching times.  A
direct comparison between the finite-$N$ RBM law and the tensorized continuous
mean field law was later obtained by Huang-Jin-Li~\cite{HuangJinLi2025}.
Under regularity and dissipativity assumptions, they proved the uniform-in-time
scaled relative-entropy estimate
\[
  \mathcal H_N\!\left(F^{N,\mathrm{RBM}}_t\mid \rho_t^{\otimes N}\right)
  = \mathcal O\!\left(\tau^2+\frac1N\right),
\]
which in turn yields a Wasserstein error of order
$\mathcal O(\tau+N^{-1/2})$.  More recently, Jin-Li-Wang used the BBGKY
hierarchy to obtain the sharper
$k$-marginal entropy estimate
$\mathcal O(k^2/N^2+k\tau^2)$~\cite{LiWangJin2025}.

Long-time behavior and time discretization of the RBM have also been studied.  Jin-Li-Ye-Zhou proved geometric ergodicity for the RBM under contractive conditions~\cite{JinLiYeZhou2023}.
Ye-Zhou also analyzed the fully time-discrete method together with its
mean field limit and obtained long-time error bounds for both particle and
McKean--Vlasov dynamics~\cite{YeZhou2024}.  For the Cucker--Smale model, Ha-Jin-Kim-Ko established a uniform-in-time error estimate based on the
flocking structure~\cite{HaJinKimKo2021}; Wang-Lin then justified the RBM mean field limit and its small-time-step limit in Wasserstein distance 
\cite{WangLin2025}.  Jin-Wang-Wang further obtained asymptotically
uniform estimates for replacement-type RBM in the Cucker--Smale
setting~\cite{JinWangWang2026}.  Other variants include the RBM for systems driven
by L\'evy processes~\cite{LiuWang2025Levy}, and the RBM with replacement~\cite{CaiLiuWang2026}.

\subsection{Proof strategy}
Our proof separates the random batch error from the mean field fluctuation. This separation is essential because the two terms use different structures of
the vortex dynamics. 

First, Proposition~\ref{prop:entropy-evolution-rbm-vortex} gives the exact
relative-entropy identity
\[
  \frac{\mathrm d}{\mathrm dt}\mathcal{H}_N(t)
  =
  -\sigma\mathcal I_N(t)
  +\mathcal R_{\mathrm{RBM}}(t)
  +\mathcal R_{\mathrm{mf}}(t).
\]
The proof therefore reduces to estimating the two remainder terms without
placing $\|\nabla K_\varepsilon\|_{L^\infty}$ in the Gronwall coefficient.

For the random batch contribution, we construct a
locally coupled auxiliary partition $\xi_{k,i}'$ for each label $i$ following an approach similar to that in \cite{LiWangJin2025}. Its conditional
unbiasedness converts the centered batch force into a difference of two
fixed-partition forces, while exchangeability gives a symmetric product of the
corresponding force and density differences. The two partitions are either
identical or differ in only two batches. Consequently, their full drift
difference is supported on at most $2p$ particle coordinates. Then we can obtain a
triangular-discrimination estimate which avoids any loss in $N$ by following a combinatorial average of the
local absolute Fisher information. Boltzmann entropy
dissipation then telescopes over the time steps and yields
\[
  \int_0^T|\mathcal R_{\mathrm{RBM}}(t)|\,\mathrm dt
  \leq
  \frac{\sigma}{8}\int_0^T\mathcal I_N(t)\,\mathrm dt
  +C_T\varepsilon^{-4}\tau^2(1+h_N^0).
\]
The factor $\varepsilon^{-4}$ comes only from
$\|K_\varepsilon\|_{L^\infty}^4\simeq\varepsilon^{-4}$.

The mean field contribution is controlled by the Biot--Savart structure.
Proposition~\ref{prop:uniform-reference-estimates} first provides Gaussian,
score, and logarithmic-Hessian estimates for $\omega_\varepsilon$ that are uniform
in the blob radius. The oddness of $K_\varepsilon$ then pairs the kernel with
the score difference
\[
  \nabla\log\omega_\varepsilon(x)
  -
  \nabla\log\omega_\varepsilon(y).
\]
The logarithmic-Hessian bound cancels the near-diagonal singularity. The
canonical exponential estimate of \cite{JabinWang2018} subsequently gives
\[
  |\mathcal R_{\mathrm{mf}}(t)|
  \leq
  C_T\left(\mathcal{H}_N(t)+\frac1N\right),
\]
with a constant independent of $\varepsilon$.

Combining the two estimates with the entropy identity and applying Gronwall's
inequality proves Theorem~\ref{thm:main}. In particular, the blob radius does
not enter the Gronwall exponent; it appears only through the algebraic
consistency factor $\varepsilon^{-4}\tau^2$.

\subsection{Organization of the paper}

Section~\ref{sec:entropy-evolution} derives the exact evolution identity for
the normalized relative entropy. Section~\ref{sec:random-batch-contribution}
constructs the locally coupled auxiliary partition and proves the integrated
random batch estimate. Section~\ref{sec:mean-field-contribution} establishes
the uniform Gaussian and logarithmic estimates for the regularized vorticity
and then controls the mean field Biot--Savart fluctuation. Finally,
Section~\ref{sec:proof-main} combines these estimates to prove
Theorem~\ref{thm:main}.

\section{Time evolution of the relative entropy}
\label{sec:entropy-evolution}
We present in this subsection the time evolution of the normalized relative entropy, which is computed
directly using the evolution PDE of the two systems. First, to deal with the time-discretized random structure of \eqref{intro-rbm-blob-liouville}, we introduce an auxiliary distribution $
    \tilde{F}^{N,\xi_k}_{\varepsilon}$ defined on time interval $\left[k\tau, (k+1)\tau\right)$ given by
\[
\tilde{F}^{N,\xi_k}_{\varepsilon}=\mathbb{E}\left[\tilde{F}^{N,\boldsymbol{\xi}}_{\varepsilon} \mid \xi_k\right] .
\]
This means that the batch at $t_k=k\tau$ is fixed to be $\xi_k$ while the randomness brought by batches before $t_k$ is averaged out. The partition $\xi_k$ is sampled independently of the particle system up to
time $t_k$. Hence
\[
  \operatorname{Law}
  \bigl(
  \tilde X_\varepsilon(t_k)
  \mid
  \xi_k
  \bigr)
  =
  \operatorname{Law}
  \bigl(
  \tilde X_\varepsilon(t_k)
  \bigr),
\]
which gives
\[
  \tilde F^{N,\xi_k}_\varepsilon(t_k)
  =
  \tilde F^N_\varepsilon(t_k).
\]
For a fixed current partition $\xi_k$, the conditional density satisfies
\begin{equation}
\label{eq:fixed-batch-fp-proof}
  \partial_t
  \tilde F^{N,\xi_k}_\varepsilon+
  \sum_{i=1}^N \operatorname{div}_{x_i}
  \left[
  \left(
    \frac1{p-1}
    \sum_{\substack{j\in\xi_k(i), j\ne i}}
    K_\varepsilon(x_i-x_j)
  \right) 
  \tilde F^{N,\xi_k}_\varepsilon \right]
 =
  \sigma
  \sum_{i=1}^N
  \Delta_{x_i}
  \tilde F^{N,\xi_k}_\varepsilon.
\end{equation}
Note that if we observe the random batch particles system \eqref{intro-rbm-blob-particle} at the discrete time $t_k$, it then forms a time-homogeneous Markov chain. It is clear that the law of $\tilde{X}_\eps^{\boldsymbol{\xi}}$ at $t_k$ is given by
\[
\tilde{F}^{N}_\varepsilon(t_k)=\mathbb{E}_{\boldsymbol{\xi}} \tilde{F}^{N,\boldsymbol{\xi}}_\varepsilon(t_k) .
\]
Hence, $\tilde{F}^{N,\xi_k}_\varepsilon$ is just the law of the random batch particles with the batch at $t_k$ prescribed. By the Markov property, we can understand that the positions at $t_k$ are drawn from $\tilde{F}^{N}_\varepsilon(t_k)$, and they evolve according to the batch $\xi_k$ during $\left[t_k, t_{k+1}\right]$. With the above understanding, it is clear that
\[
\tilde{F}^{N}_\varepsilon=\mathbb{E}_{\xi_k} \tilde{F}^{N,\xi_k}_\varepsilon=\mathbb{E}_{\boldsymbol{\xi}} \tilde{F}^{N,\boldsymbol{\xi}}_\varepsilon,
\]
for $t \in\left[t_k, t_{k+1}\right)$. Then one has
\begin{equation}\label{rbm-blob-liouville-pde-average}
    \partial_t \tilde{F}^{N}_\varepsilon+ \sum_{i=1}^N \mathbb{E}_{\xi_k} \left[\operatorname{div}_{x_i}
    \left(\left(\frac{1}{p-1} \sum_{\substack{j\in\xi_k(i),  j\ne i}}K_\eps\left(x_i-x_j\right)\right) \cdot \nabla_{x_i} \tilde{F}^{N,\xi_k}_\varepsilon\right)\right]=\sigma \sum_{i=1}^N \Delta_{x_i} \tilde{F}^{N}_\varepsilon.
\end{equation}

\begin{proposition}[Time evolution of the normalized relative entropy]
\label{prop:entropy-evolution-rbm-vortex}
Let $N\geq2$, let $2\leq p\leq N$ with $p\mid N$, and fix
$k\geq0$ and $t\in(t_k,t_{k+1})$, where $t_k=k\tau$.
For a prescribed current random batch partition $\xi_k$, let
$\tilde F^{N,\xi_k}_\varepsilon(t,X^N)$ be the conditional law of the
random batch system on $[t_k,t_{k+1})$, with
\[
  \tilde F^{N,\xi_k}_\varepsilon(t_k)
  =
  \tilde F^N_\varepsilon(t_k),
  \qquad
  \tilde F^N_\varepsilon(t)
  =
  \mathbb E_{\xi_k}
  \tilde F^{N,\xi_k}_\varepsilon(t).
\]
Recall that
\[
  \mathcal H_N
  \bigl(
  \tilde F^N_\varepsilon
  \mid
  \omega_\varepsilon^{\otimes N}
  \bigr)
  :=
  \frac1N
  \int_{\mathbb R^{2N}}
  \tilde F^N_\varepsilon
  \log
  \frac{
    \tilde F^N_\varepsilon
  }{
    \omega_\varepsilon^{\otimes N}
  }
  \,\mathrm dX^N.
\]
Define the normalized relative Fisher information by
\begin{equation}
\label{eq:relative-fisher-rbm-vortex}
  \mathcal I_N(t)
  :=
  \frac1N
  \sum_{i=1}^N
  \int_{\mathbb R^{2N}}
  \tilde F^N_\varepsilon
  \left|
  \nabla_{x_i}
  \log
  \frac{
    \tilde F^N_\varepsilon
  }{
    \omega_\varepsilon^{\otimes N}
  }
  \right|^2
  \,\mathrm dX^N.
\end{equation}
Then the time evolution of the normalized relative entropy can be written as
\begin{equation}\label{eq:entropy-evolution-rbm-vortex}
   \frac{\mathrm{d}}{\mathrm{~d} t} \mathcal{H}_N\left(\tilde{F}_{\varepsilon}^N \mid \omega_{\varepsilon}^{\otimes N}\right)=-\sigma \mathcal{I}_N(t)+\mathcal{R}_{\mathrm{RBM}}(t)+\mathcal{R}_{\mathrm{mf}}(t), 
\end{equation}
where
\begin{align}
\label{eq:R-rbm-linear}
\mathcal R_{\mathrm{RBM}}(t)
={}&
\frac1N
\sum_{i=1}^N
\mathbb E_{\xi_k}
\int_{\mathbb R^{2N}}
\tilde F^{N,\xi_k}_\varepsilon
\Bigg[
  \frac1{p-1}
  \sum_{\substack{j\in\xi_k(i), j\ne i}}
  K_\varepsilon(x_i-x_j)
  -
  \frac1{N-1}
  \sum_{j\ne i}
  K_\varepsilon(x_i-x_j)
\Bigg]
\notag\\
&\hspace{3.2cm}
\cdot
\nabla_{x_i}
\log
\frac{
  \tilde F^N_\varepsilon
}{
  \omega_\varepsilon^{\otimes N}
}
\,\mathrm dX^N,
\end{align}
and
\begin{align}
\label{eq:R-mf-score}
\mathcal R_{\mathrm{mf}}(t)
=-{}&
\frac1N
\sum_{i=1}^N
\int_{\mathbb R^{2N}}
\tilde F^N_\varepsilon
\Bigg[
  \frac1{N-1}
  \sum_{j\ne i}
  K_\varepsilon(x_i-x_j)
  -
  (K_\varepsilon*\omega_\varepsilon)(x_i)
\Bigg]
\cdot
\nabla\log\omega_\varepsilon(x_i)
\,\mathrm dX^N.
\end{align}

\end{proposition}

\begin{proof}
Since both densities have unit mass, differentiation of the normalized
relative entropy yields
\begin{equation}\label{eq:entropy-derivative-first}
\frac{\mathrm{d}}{\mathrm{~d} t} \mathcal{H}_N\left(\tilde{F}_{\varepsilon}^N \mid \omega_{\varepsilon}^{\otimes N}\right)=\frac{1}{N} \int_{\mathbb{R}^{2 N}} \partial_t \tilde{F}_{\varepsilon}^N \log \frac{\tilde{F}_{\varepsilon}^N}{\omega_{\varepsilon}^{\otimes N}} \mathrm{~d} X^N-\frac{1}{N} \int_{\mathbb{R}^{2 N}} \tilde{F}_{\varepsilon}^N \frac{\partial_t \omega_{\varepsilon}^{\otimes N}}{\omega_{\varepsilon}^{\otimes N}} \mathrm{~d} X^N.
\end{equation}
Recall the joint law $\tilde F^N_\varepsilon$ and tensorized law $\omega_\varepsilon^{\otimes N}$ satisfy \eqref{rbm-blob-liouville-pde-average} and \eqref{intro-tensorized-pde}, respectively. Substituting \eqref{rbm-blob-liouville-pde-average} and \eqref{intro-tensorized-pde} into
\eqref{eq:entropy-derivative-first}, and integrating by parts, gives the
transport contribution
\begin{align}
\label{eq:transport-before-combination}
&
\frac1N
\sum_{i=1}^N
\int_{\mathbb R^{2N}}
\mathbb E_{\xi_k}
\left[
\tilde F^{N,\xi_k}_\varepsilon
\left(
  \frac1{p-1}
  \sum_{\substack{j\in\xi_k(i), j\ne i}}
  K_\varepsilon(x_i-x_j)
\right)
\right]
\cdot
\nabla_{x_i}
\log
\frac{
  \tilde F^N_\varepsilon
}{
  \omega_\varepsilon^{\otimes N}
}
\,\mathrm dX^N
\notag\\
&\qquad
+
\frac1N
\sum_{i=1}^N
\int_{\mathbb R^{2N}}
\tilde F^N_\varepsilon
(K_\varepsilon*\omega_\varepsilon)(x_i)
\cdot
\nabla_{x_i}
\log
\omega_\varepsilon^{\otimes N}
\,\mathrm dX^N.
\end{align}
Because $
  \operatorname{div}
  (K_\varepsilon*\omega_\varepsilon)
  =
  0, $
one has
\begin{align*}
\int_{\mathbb R^{2N}}
\tilde F^N_\varepsilon
(K_\varepsilon*\omega_\varepsilon)(x_i)
\cdot
\nabla_{x_i}
\log
\tilde F^N_\varepsilon
\,\mathrm dX^N
=
\int_{\mathbb R^{2N}}
(K_\varepsilon*\omega_\varepsilon)(x_i)
\cdot
\nabla_{x_i}
\tilde F^N_\varepsilon
\,\mathrm dX^N
=
0.
\end{align*}
Consequently, \eqref{eq:transport-before-combination} equals
\begin{equation}\label{eq:transport-combined}
    \frac{1}{N} \sum_{i=1}^N \int_{\mathbb{R}^{2 N}} \mathbb{E}_{\xi_k}\left[\tilde{F}_{\varepsilon}^{N, \xi_k}\left(\frac{1}{p-1} \sum_{\substack{j \in \xi_k(i), j \neq i}} K_{\varepsilon}\left(x_i-x_j\right)-\left(K_{\varepsilon} * \omega_{\varepsilon}\right)\left(x_i\right)\right)\right] \cdot \nabla_{x_i} \log \frac{\tilde{F}_{\varepsilon}^N}{\omega_{\varepsilon}^{\otimes N}} \mathrm{~d} X^N.
\end{equation}
For completeness, the diffusion contribution is
\begin{align*}-
\frac{\sigma}{N}
\sum_{i=1}^N
\int_{\mathbb R^{2N}}
\nabla_{x_i}
\tilde F^N_\varepsilon
\cdot
\nabla_{x_i}
\log
\frac{
  \tilde F^N_\varepsilon
}{
  \omega_\varepsilon^{\otimes N}
}
\,\mathrm dX^N
-
\frac{\sigma}{N}
\sum_{i=1}^N
\int_{\mathbb R^{2N}}
\tilde F^N_\varepsilon
\frac{
  \Delta_{x_i}
  \omega_\varepsilon^{\otimes N}
}{
  \omega_\varepsilon^{\otimes N}
}
\,\mathrm dX^N.
\end{align*}
Using
\[
  \frac{
    \Delta_{x_i}
    \omega_\varepsilon^{\otimes N}
  }{
    \omega_\varepsilon^{\otimes N}
  }
  =
  \Delta_{x_i}
  \log
  \omega_\varepsilon^{\otimes N}
  +
  \left|
  \nabla_{x_i}
  \log
  \omega_\varepsilon^{\otimes N}
  \right|^2,
\]
and integrating the Laplacian term by parts, the diffusion contribution
becomes
\[
  -
  \frac{\sigma}{N}
  \sum_{i=1}^N
  \int_{\mathbb R^{2N}}
  \tilde F^N_\varepsilon
  \left|
  \nabla_{x_i}
  \log
  \frac{
    \tilde F^N_\varepsilon
  }{
    \omega_\varepsilon^{\otimes N}
  }
  \right|^2
  \,\mathrm dX^N
  =
  -\sigma\mathcal I_N(t).
\]
Combining this identity with \eqref{eq:transport-combined} gives the exact
entropy identity before decomposition.

Insert $\frac1{N-1}
  \sum_{j\ne i}
  K_\varepsilon(x_i-x_j)$ between the batch force and the mean field force in
\eqref{eq:transport-combined}. Since $
  \mathbb E_{\xi_k}
  \tilde F^{N,\xi_k}_\varepsilon
  =
  \tilde F^N_\varepsilon$, this gives
\[
  \frac{\mathrm d}{\mathrm dt}
  \mathcal H_N
  \bigl(
  \tilde F^N_\varepsilon
  \mid
  \omega_\varepsilon^{\otimes N}
  \bigr)
  =
  -\sigma\mathcal I_N(t)
  +
  \mathcal R_{\mathrm{RBM}}(t)
  +
  \mathcal R_{\mathrm{mf}}(t),
\]
where
\begin{align*}
\mathcal R_{\mathrm{RBM}}(t)
={}&
\frac1N
\sum_{i=1}^N
\mathbb E_{\xi_k}
\int_{\mathbb R^{2N}}
\tilde F^{N,\xi_k}_\varepsilon
\Bigg[
  \frac1{p-1}
  \sum_{\substack{j\in\xi_k(i), j\ne i}}
  K_\varepsilon(x_i-x_j)
  -
  \frac1{N-1}
  \sum_{j\ne i}
  K_\varepsilon(x_i-x_j)
\Bigg]
\notag\\
&\hspace{3.2cm}
\cdot
\nabla_{x_i}
\log
\frac{
  \tilde F^N_\varepsilon
}{
  \omega_\varepsilon^{\otimes N}
}
\,\mathrm dX^N,
\end{align*}
and
\begin{align}
\label{eq:R-mf-linear}
\mathcal R_{\mathrm{mf}}(t)
={}&
\frac1N
\sum_{i=1}^N
\int_{\mathbb R^{2N}}
\tilde F^N_\varepsilon
\Bigg[
  \frac1{N-1}
  \sum_{j\ne i}
  K_\varepsilon(x_i-x_j)
  -
  (K_\varepsilon*\omega_\varepsilon)(x_i)
\Bigg]
\notag\\
&\hspace{3.2cm}
\cdot
\nabla_{x_i}
\log
\frac{
  \tilde F^N_\varepsilon
}{
  \omega_\varepsilon^{\otimes N}
}
\,\mathrm dX^N.
\end{align}
Finally, the vector field in square brackets in
\eqref{eq:R-mf-linear} is divergence-free with respect to $x_i$. Hence its
contribution against
$\nabla_{x_i}\log\tilde F^N_\varepsilon$ vanishes by integration by parts.
Since $
  \nabla_{x_i}
  \log
  \omega_\varepsilon^{\otimes N}(X^N)
  =
  \nabla
  \log
  \omega_\varepsilon(x_i),$
this gives \eqref{eq:R-mf-score}.
\end{proof}

\section{The random batch contribution}
\label{sec:random-batch-contribution}

In this section, we control the random batch term $\mathcal{R}_{\mathrm{RBM}}$ \eqref{eq:R-rbm-linear} in the relative-entropy identity. Throughout this section, set
\(2\le p\le N\), \(p\mid N\), and \(t_k=k\tau\). For the random batch partition $\xi_k$ at time $t_k$, let \(\xi_k(i)\) be the unique batch containing
\(i\), and for notational convenience, define
\begin{equation}
\label{eq:rbm-batch-drift}
 B_i^{\xi_k}(X^N)
 :=
 \frac1{p-1}
 \sum_{\substack{j\in\xi_k(i)\\j\ne i}}
 K_\varepsilon(x_i-x_j),
 \qquad
 \overline B_i(X^N)
 :=
 \frac1{N-1}\sum_{j\ne i}K_\varepsilon(x_i-x_j).
\end{equation}
We also write
\[
 \boldsymbol B^{\xi_k}:=(B_1^{\xi_k},\ldots,B_N^{\xi_k}),
 \qquad
 |\boldsymbol v|^2:=\sum_{i=1}^N|v_i|^2,
\]
and
\[
 f_t^{\xi_k}:=\widetilde F_\varepsilon^{N,\xi_k}(t),
 \qquad
 f_t:=\widetilde F_\varepsilon^N(t)
      =\mathbb E_{\xi_k}f_t^{\xi_k}.
\]
Then the Liouville equation \eqref{eq:fixed-batch-fp-proof} can be written as:
\begin{equation}
\label{eq:rbm-fixed-partition-fp}
 \partial_t f_t^{\xi_k}
 +
 \operatorname{div}_{X^N}(\boldsymbol B^{\xi_k} f_t^{\xi_k})
 =
 \sigma\Delta_{X^N}f_t^{\xi_k},
 \qquad
 f_{t_k}^{\xi_k}=f_{t_k}.
\end{equation}
Here \(\operatorname{div}_{X^N}\) and \(\Delta_{X^N}\) denote the divergence and Laplacian in the \(2N\)-dimensional space of particle positions \(X^N=(x_1,\ldots,x_N)\).
Since \(\operatorname{div}K_\varepsilon=0\),
\begin{equation}
\label{eq:rbm-full-drift-div-free}
 \operatorname{div}_{X^N}\boldsymbol B^{\xi_k}=0
 \qquad
 \text{for every partition } \xi_k.
\end{equation}
Then the random batch contribution \eqref{eq:R-rbm-linear} obtained from the entropy identity is
\begin{equation}
\label{eq:rbm-centered-term}
 \mathcal R_{\mathrm{RBM}}(t)
 =
 \frac1N\sum_{i=1}^N\mathbb E_{\xi_k}
 \int_{\mathbb R^{2N}}
 f_t^{\xi_k}
 \bigl(B_i^{\xi_k}-\overline B_i\bigr)
 \cdot  \nabla_{x_i}  \log  \frac{f_t}  {\omega_\varepsilon^{\otimes N}}\,\mathrm dX^N.
\end{equation}

\subsection{A locally coupled auxiliary partition}

The main device to bound term \eqref{eq:rbm-centered-term} is a local coupling on the space of batch partitions which is similar to \cite{LiWangJin2025}. For
each fixed particle label \(i\), the auxiliary partition is either left
unchanged or differs from the current partition only in the two batches used
in the exchange. We can prove that the auxiliary partition has the same marginal law as the
current partition.
Let \(\mathfrak P_{N,p}\) denote the set of partitions of
\(\{1,\ldots,N\}\) into batches of size \(p\), and let \(\mu_{N,p}\) be the
uniform probability measure on \(\mathfrak P_{N,p}\). For each fixed \(i\),
we now construct an auxiliary full partition \(\xi_{k,i}'\). The construction
depends on \(i\); in particular, the terms in the sum over \(i\) do not use a
single common auxiliary partition.

\begin{lemma}[Local coupling of batch partitions]
\label{lem:rbm-local-partition-coupling}
Assume \(2\le p<N\). Fix \(i\in\{1,\ldots,N\}\). Let
\(\xi_k\sim\mu_{N,p}\), choose
$
 U\sim
 \operatorname{Unif}
 \bigl(\{1,\ldots,N\}\setminus\{i\}\bigr)
$
independently of \(\xi_k\), and define
\(\xi_{k,i}'\) as follows:
\begin{itemize}
  \item If \(U\in\xi_k(i)\setminus\{i\}\), set \(\xi_{k,i}'=\xi_k\).
  \item If \(U\notin\xi_k(i)\), let \(S=\xi_k(i)\) and \(C=\xi_k(U)\) be the two batches containing \(i\) and \(U\). 
  Replace \(S\) and \(C\) by 
\begin{equation}\label{eq:rbm-two-batch-exchange} 
 S':=\{i\}\cup(C\setminus\{U\}),
 \qquad
 C':=\{U\}\cup(S\setminus\{i\}).
\end{equation}
Equivalently,
\[
 \xi_{k,i}'
 =
 \bigl(\xi_k\setminus\{S,C\}\bigr)
 \cup
 \{S',C'\}.
\]
\end{itemize}
Then the following statements hold.
\begin{enumerate}
\item The pair is exchangeable:
\begin{equation}
\label{eq:rbm-partition-exchangeability}
  (\xi_k,\xi_{k,i}')
  \stackrel{\mathrm d}{=}
  (\xi_{k,i}',\xi_k).
\end{equation}
In particular, \(\xi_{k,i}'\sim\mu_{N,p}\).

\item For every \(\ell\ne i\),
\begin{equation}
\label{eq:rbm-conditional-inclusion}
 \mathbb P
 \bigl(
 \ell\in\xi_{k,i}'(i)\mid\xi_k
 \bigr)
 =
 \frac{p-1}{N-1}.
\end{equation}
Consequently,
\begin{equation}
\label{eq:rbm-conditional-unbiasedness}
 \mathbb E
 \bigl[
 B_i^{\xi_{k,i}'}\mid\xi_k
 \bigr]
 =
 \overline B_i.
\end{equation}

\item On the no-exchange event, set
\[
 \Lambda_i:=\varnothing.
\]
On the exchange event, set
\begin{equation}
\label{eq:rbm-affected-label-set}
 \Lambda_i
 :=
 \xi_k(i)\cup\xi_k(U)
 =
 \xi_{k,i}'(i)\cup\xi_{k,i}'(U).
\end{equation}
Then \(|\Lambda_i|=2p\) on the exchange event. Moreover,
\begin{equation}
\label{eq:rbm-drift-localization}
 B_a^{\xi_k}=B_a^{\xi_{k,i}'}
 \qquad
 \text{for every }a\notin\Lambda_i,
\end{equation}
and 
\begin{equation}
\label{eq:rbm-full-drift-difference}
 \bigl|
 \boldsymbol B^{\xi_k}
 -
 \boldsymbol B^{\xi_{k,i}'}
 \bigr|^2
 \le
 8pM_\varepsilon^2,
\end{equation}
with $
 M_\varepsilon
 :=
 \|K_\varepsilon\|_{L^\infty(\mathbb R^2)}.
$
\item For every fixed partition
\(\xi_k\) and every choice of nonnegative numbers
\(A_1,\ldots,A_N\),
\begin{equation}
\label{eq:rbm-combinatorial-count}
 \sum_{i=1}^N
 \mathbb E
 \left[
 \left.
 \sum_{a\in\Lambda_i}A_a
 \,\right|\,
 \xi_k
 \right]
 =
 \frac{2p(N-p)}{N-1}
 \sum_{a=1}^NA_a
 \le
 2p\sum_{a=1}^NA_a.
\end{equation}
The same identity remains valid when  \(A_a\) are nonnegative functions
of the fixed partition \(\xi_k\).
\end{enumerate}
\end{lemma}

\begin{proof}
For fixed \(u\ne i\), denote by \(T_{i,u}\) the map on
\(\mathfrak P_{N,p}\) defined by the construction above. The map is an
involution: applying the same exchange twice restores the original two
batches, while on the no-exchange event it is the identity. Hence, for every
bounded function \(\Phi\),
\begin{align*}
 &\mathbb E\,\Phi(\xi_k,\xi_{k,i}',U)
 \\
 &\quad=
 \frac{1}{|\mathfrak P_{N,p}|(N-1)}
 \sum_{\xi\in\mathfrak P_{N,p}}\sum_{u\ne i}
 \Phi\bigl(\xi,T_{i,u}\xi,u\bigr)
 \\
 &\quad=
 \frac{1}{|\mathfrak P_{N,p}|(N-1)}
 \sum_{\eta\in\mathfrak P_{N,p}}\sum_{u\ne i}
 \Phi\bigl(T_{i,u}\eta,\eta,u\bigr)
 \\
 &\quad=
 \mathbb E\,\Phi(\xi_{k,i}',\xi_k,U).
\end{align*}
This proves \eqref{eq:rbm-partition-exchangeability} and the equality of the
marginal laws.

Fix \(\ell\ne i\) and condition on \(\xi_k\). If $
 \ell\in\xi_k(i)\setminus\{i\}$, then \(\ell\) remains in the batch of \(i\) exactly when
$
 U\in\xi_k(i)\setminus\{i\},
$
which gives \(p-1\) admissible choices of \(U\).
If $
 \ell\notin\xi_k(i)$, then \(\ell\) belongs to \(\xi_{k,i}'(i)\) exactly when
$
 U\in\xi_k(\ell)\setminus\{\ell\}$, again giving \(p-1\) admissible choices. Since \(U\) is uniform over
\(N-1\) labels, \eqref{eq:rbm-conditional-inclusion} follows. 

Moreover, pointwise in \(X^N\) one has
\begin{align*}
 \mathbb E
 \bigl[
 B_i^{\xi_{k,i}'}\mid\xi_k
 \bigr]
 &=
 \frac1{p-1}
 \sum_{\ell\ne i}
 \mathbb P
 \bigl(
 \ell\in\xi_{k,i}'(i)\mid\xi_k
 \bigr)
 K_\varepsilon(x_i-x_\ell)
 \\
 &=
 \frac1{N-1}
 \sum_{\ell\ne i}
 K_\varepsilon(x_i-x_\ell)
 \\
 &=
 \overline B_i,
\end{align*}
which proves \eqref{eq:rbm-conditional-unbiasedness}.

If no exchange occurs, all drift components agree. Otherwise, only the two
batches in \eqref{eq:rbm-two-batch-exchange} are modified, so
\eqref{eq:rbm-drift-localization} holds. Moreover, for every partition
\(\xi_k\) and every label \(a\),
\[
 |B_a^{\xi_k}|
 \le
 \frac1{p-1}
 \sum_{\substack{b\in\xi_k(a)\\b\ne a}}
 |K_\varepsilon(x_a-x_b)|
 \le
 M_\varepsilon.
\]
Thus $
 |B_a^{\xi_k}-B_a^{\xi_{k,i}'}|
 \le
 2M_\varepsilon$ on \(\Lambda_i\), and
\[
 \bigl|
 \boldsymbol B^{\xi_k}
 -
 \boldsymbol B^{\xi_{k,i}'}
 \bigr|^2
 =
 \sum_{a\in\Lambda_i}
 |B_a^{\xi_k}-B_a^{\xi_{k,i}'}|^2
 \le
 2p(2M_\varepsilon)^2,
\]
which is \eqref{eq:rbm-full-drift-difference}.

It remains to prove \eqref{eq:rbm-combinatorial-count}. Fix a partition
\(\xi_k\), the no-exchange event gives zero because
\(\Lambda_i=\varnothing\). Pick one \(j\notin\xi_k(i)\) with
probability \(1/(N-1)\), hence we have the following decomposition
\begin{align*}
 \sum_{i=1}^N
 \mathbb E
 \left[
 \left.
 \sum_{a\in\Lambda_i}A_a
 \,\right|\,
 \xi_k
 \right] =
 \frac1{N-1}
 \sum_{i=1}^N
 \sum_{j\notin\xi_k(i)}
 \left(
 \sum_{a\in\xi_k(i)}A_a
 +
 \sum_{a\in\xi_k(j)}A_a
 \right).
\end{align*}
Note that in each of the two sums in parentheses, a fixed label \(a\) is counted
\(p(N-p)\) times. Indeed, one chooses one label in the batch containing \(a\) in
\(p\) ways and one label outside that batch in \(N-p\) ways. Therefore the
last display equals
\[
 \frac{2p(N-p)}{N-1}
 \sum_{a=1}^NA_a.
\]
This proves \eqref{eq:rbm-combinatorial-count}.
\end{proof}
The above fact about the constructed  auxiliary partition $\xi_{k,i}^\prime$ is crucial to deal with the random batch term $\cR_{\mathrm{RBM}}$.
Use the tower property, one can rewrite the random batch term \eqref{eq:rbm-centered-term} as follows:
\begin{lemma}[Symmetric representation of the random batch term]
\label{lem:rbm-symmetric-representation}
For each fixed \(i\), let \(\xi_{k,i}'\) be the locally coupled
auxiliary partition constructed above, and denote by
$
  \mathbb E_i
  :=
  \mathbb E_{\xi_k}
  \mathbb E_{\xi_{k,i}'\mid\xi_k}
$
expectation with respect to its coupled joint law. Then
\begin{align}
\label{eq:rbm-symmetric-representation}
 \mathcal R_{\mathrm{RBM}}(t)
 ={}&
 \frac1{2N}
 \sum_{i=1}^N
 \mathbb E_i
 \int_{\mathbb R^{2N}}
 \bigl(
 f_t^{\xi_k}
 -
 f_t^{\xi_{k,i}'}
 \bigr)
 \bigl(
 B_i^{\xi_k}
 -
 B_i^{\xi_{k,i}'}
 \bigr)\cdot
 \nabla_{x_i}
 \log
 \frac{f_t}
 {\omega_\varepsilon^{\otimes N}}
 \,\mathrm dX^N.
\end{align}
If \(p=N\), then \(B_i^{\xi_k}=\overline B_i\) and
\(\mathcal R_{\mathrm{RBM}}\equiv0\).
\end{lemma}

\begin{proof}
For each fixed \(i\), the conditional unbiasedness
\eqref{eq:rbm-conditional-unbiasedness} and the tower property give
\begin{align*}
 \mathcal R_{\mathrm{RBM}}(t)
 ={}&
 \frac1N\sum_{i=1}^N
 \mathbb E_i
 \int_{\mathbb R^{2N}}
 f_t^{\xi_k}
 \bigl(
 B_i^{\xi_k}
 -
 B_i^{\xi_{k,i}'}
 \bigr)
 \cdot
 \nabla_{x_i}
 \log
 \frac{f_t}
 {\omega_\varepsilon^{\otimes N}}
 \,\mathrm dX^N.
\end{align*}
By \eqref{eq:rbm-partition-exchangeability}, the pair
\((\xi_k,\xi_{k,i}')\) is exchangeable. Therefore,
\begin{align*}
 &\mathbb E_i
 \int
 f_t^{\xi_k}
 \bigl(
 B_i^{\xi_k}
 -
 B_i^{\xi_{k,i}'}
 \bigr)
 \cdot \nabla_{x_i}\log
 \frac{f_t}
 {\omega_\varepsilon^{\otimes N}}\,\mathrm dX^N=
 \mathbb E_i
 \int
 f_t^{\xi_{k,i}'}
 \bigl(
 B_i^{\xi_{k,i}'}
 -
 B_i^{\xi_k}
 \bigr)
 \cdot \nabla_{x_i}\log
 \frac{f_t}
 {\omega_\varepsilon^{\otimes N}}\,\mathrm dX^N.
\end{align*}
Averaging these two representations gives
\eqref{eq:rbm-symmetric-representation}.
\end{proof}

\subsection{Estimate of the integrated random batch term}
Before stating the estimate, we introduce some notation. For a probability density \(F\) on \(\mathbb R^{2N}\), define the normalized
Boltzmann entropy, second moment, and absolute Fisher information by
\begin{align}
 \mathcal S_N(F)
 &:=
 \frac1N
 \int_{\mathbb R^{2N}}F\log F\,\mathrm dX^N,
 \label{eq:rbm-Boltzmann-entropy}
 \\
 \mathcal M_{2,N}(F)
 &:=
 \frac1N
 \int_{\mathbb R^{2N}}
 \sum_{a=1}^N|x_a|^2F\,\mathrm dX^N,
 \label{eq:rbm-second-moment}
 \\
 \mathcal J_N(F)
 &:=
 \frac1N
 \sum_{a=1}^N
 \int_{\mathbb R^{2N}}
 F|\nabla_{x_a}\log F|^2\,\mathrm dX^N.
 \label{eq:rbm-absolute-fisher}
\end{align}
For \(\Lambda\subset\{1,\ldots,N\}\), also set
\begin{equation}
\label{eq:rbm-local-fisher}
 \mathcal J_\Lambda(F)
 :=
 \sum_{a\in\Lambda}
 \int_{\mathbb R^{2N}}
 F|\nabla_{x_a}\log F|^2\,\mathrm dX^N.
\end{equation}

\begin{proposition}[Integrated random batch estimate]
\label{prop:rbm-integrated-estimate}
Assume the initial density $f_0:=\widetilde F_\varepsilon^N(0)$ satisfies
\[
 \mathcal S_N(f_0)<\infty,
 \qquad
 \mathcal M_{2,N}(f_0)<\infty.
\]
Then, for the mollified Biot--Savart kernel
$
 K_\varepsilon=K*\varphi_\varepsilon
$ and every \(T>0\), we have
\begin{align}
\label{eq:rbm-integrated-general}
 \int_0^T
 |\mathcal R_{\mathrm{RBM}}(t)|\,\mathrm dt
 \le{}&
 \frac\sigma8
 \int_0^T\mathcal I_N(t)\,\mathrm dt
 \notag\\
 &+
 \frac{Cp^2}{\sigma^2}
 M_\varepsilon^{4}\tau^2
 \Bigl[
 \mathcal S_N(f_0)
 +\log\pi
 +\mathcal M_{2,N}(f_0)
 +(A_\varphi+4\sigma)T
 \Bigr],
\end{align}
where \(C>0\) is universal and the estimate is independent of \(N\).
Here \begin{equation}
\label{eq:rbm-kernel-scaling}
 M_\varepsilon
 :=
 \varepsilon^{-1}
 \|K*\varphi\|_{L^\infty(\mathbb R^2)}, \quad  A_\varphi
 :=
 \sup_{0<\varepsilon\le1}
 \sup_{z\in\mathbb R^2}
 |z|\,|K_\varepsilon(z)|
 <\infty.
\end{equation}
In addition, if \(\omega_0\in L^\infty(\mathbb R^2)\) is strictly positive and there exist \(b_0,C_0>0\) such
that
\[
 \int_{\mathbb R^2}
 e^{b_0|x|^2}\omega_0(x)\,\mathrm dx
 <\infty,
 \qquad
 -\log\omega_0(x)
 \le
 C_0(1+|x|^2).
\]
Then
\begin{equation}
\label{eq:rbm-integrated-relative-initial}
 \int_0^T
 |\mathcal R_{\mathrm{RBM}}(t)|\,\mathrm dt
 \le
 \frac\sigma8
 \int_0^T\mathcal I_N(t)\,\mathrm dt
 +
 C_{p,\sigma,\varphi,T,\omega_0}
 \varepsilon^{-4}\tau^2(1+\mathcal H_N(f_0\mid\omega_0^{\otimes N})).
\end{equation}
\end{proposition}

\begin{proof}
The case \(p=N\) is trivial by
Lemma~\ref{lem:rbm-symmetric-representation}, so assume \(p<N\). We first
prove the estimate for smooth, strictly positive densities with sufficient
decay. The approximation argument is given at the end.

\medskip
\noindent
\textbf{Step 1: reduction to triangular discrimination.}
Fix \(t\in[t_k,t_{k+1})\), using Young's inequality for \eqref{eq:rbm-symmetric-representation} one has
\begin{align}
\label{eq:rbm-young}
 |\mathcal R_{\mathrm{RBM}}(t)|
 \le{}&
 \frac\sigma{16N}
 \sum_{i=1}^N
 \mathbb E_i
 \int_{\mathbb R^{2N}}
 (f_t^{\xi_k}+f_t^{\xi_{k,i}'})| \nabla_{x_i}  \log  \frac{f_t}  {\omega_\varepsilon^{\otimes N}}|^2\,\mathrm dX^N  \notag\\
 &+
 \frac1{\sigma N}
 \sum_{i=1}^N
 \mathbb E_i
 \int_{\mathbb R^{2N}}
 \frac{|f_t^{\xi_k}-f_t^{\xi_{k,i}'}|^2}{f_t^{\xi_k}+f_t^{\xi_{k,i}'}}|B_i^{\xi_k}-B_i^{\xi_{k,i}'}|^2\,\mathrm dX^N.
\end{align}
By Lemma~\ref{lem:rbm-local-partition-coupling}, $\xi_k$ and $\xi_{k,i}'$ have the same uniform marginal law. Hence, for every fixed \(i\),
\[
  \mathbb E_i f_t^{\xi_k}
  =
  \mathbb E_i f_t^{\xi_{k,i}'}
  =
  \mathbb E_{\xi_k}f_t^{\xi_k}
  =
  f_t.
\]
Hence, pointwise in \(X^N\),
\[
  \mathbb E_i
  \left(
    f_t^{\xi_k}
    +
    f_t^{\xi_{k,i}'}
  \right)
  =
  2f_t.
\]
Therefore, the first term on the right-hand side of \eqref{eq:rbm-young} is thus
$
 \frac\sigma8\mathcal I_N(t).
$\\
Moreover, since $
 |B_i^{\xi_k}-B_i^{\xi_{k,i}'}|\le2M_\varepsilon$, then
\begin{equation}
\label{eq:rbm-by-D}
 |\mathcal R_{\mathrm{RBM}}(t)|
 \le
 \frac\sigma8\mathcal I_N(t)
 +
 \frac{4M_\varepsilon^2}{\sigma N}
 \sum_{i=1}^N
 \mathbb E_i
 \int_{\mathbb R^{2N}}
 \frac{
 |f_t^{\xi_k}-f_t^{\xi_{k,i}'}|^2
 }{
 f_t^{\xi_k}+f_t^{\xi_{k,i}'}
 }
 \,\mathrm dX^N.
\end{equation}
\medskip
\noindent
\textbf{Step 2: stability of two locally coupled fixed-partition laws.}

Fix \(i,\xi_k,\xi_{k,i}'\), define
\begin{equation}
\label{eq:rbm-triangular-discrimination}
 \mathscr D_i(t;\xi_k,\xi_{k,i}')
 :=
 \int_{\mathbb R^{2N}}
 \frac{
 |f_t^{\xi_k}-f_t^{\xi_{k,i}'}|^2
 }{
 f_t^{\xi_k}+f_t^{\xi_{k,i}'}
 }
 \,\mathrm dX^N.
\end{equation} 
Using the Liouville equations \eqref{eq:rbm-fixed-partition-fp}, a direct
calculation gives
\begin{align}
\label{eq:rbm-D-exact}
 \frac{\mathrm d}{\mathrm dt}
 \mathscr D_i(t;\xi_k,\xi_{k,i}')
 ={}&2\int_{\mathbb{R}^{2 N}}  \frac{
 f_t^{\xi_k}-f_t^{\xi_{k,i}'}
 }{
 f_t^{\xi_k}+f_t^{\xi_{k,i}'}
 }\left(\partial_t f_t^{\xi_k}-\partial_t f_t^{\xi_{k,i}'}\right) \mathrm{d} X^N
 -\int_{\mathbb{R}^{2 N}}\left|\frac{
 f_t^{\xi_k}-f_t^{\xi_{k,i}'}
 }{
 f_t^{\xi_k}+f_t^{\xi_{k,i}'}
 }\right|^2\left(\partial_t f_t^{\xi_k}+\partial_t f_t^{\xi_{k,i}'}\right) \mathrm{d} X^N
  \notag\\
 =&
 -2\sigma
 \int_{\mathbb R^{2N}}
 \left( f_t^{\xi_k}+f_t^{\xi_{k,i}'}\right)\left|\nabla_{X^N} \frac{
 f_t^{\xi_k}-f_t^{\xi_{k,i}'}
 }{
 f_t^{\xi_k}+f_t^{\xi_{k,i}'}
 }\right|^2\,\mathrm dX^N
 \notag\\
 &+
 4
 \int_{\mathbb R^{2N}}
 \frac{
 f_t^{\xi_k}f_t^{\xi_{k,i}'}
 }{
 f_t^{\xi_k}+f_t^{\xi_{k,i}'}
 }
 \bigl(
 \boldsymbol B^{\xi_k}
 -
 \boldsymbol B^{\xi_{k,i}'}
 \bigr)
 \cdot
 \nabla_{X^N}\frac{
 f_t^{\xi_k}-f_t^{\xi_{k,i}'}
 }{
 f_t^{\xi_k}+f_t^{\xi_{k,i}'}
 }
 \,\mathrm dX^N.
\end{align}
Recall that both drifts are divergence free by \eqref{eq:rbm-full-drift-div-free}, hence
\begin{align*}
 &\int
  \frac{
 f_t^{\xi_k}f_t^{\xi_{k,i}'}
 }{
 f_t^{\xi_k}+f_t^{\xi_{k,i}'}
 } \bigl(
 \boldsymbol B^{\xi_k}
 -
 \boldsymbol B^{\xi_{k,i}'}
 \bigr)\cdot\nabla_{X^N}\frac{
 f_t^{\xi_k}-f_t^{\xi_{k,i}'}
 }{
 f_t^{\xi_k}+f_t^{\xi_{k,i}'}
 }\,\mathrm dX^N\\
 =& -
 \int
 \frac{
 f_t^{\xi_k}-f_t^{\xi_{k,i}'}
 }{
 f_t^{\xi_k}+f_t^{\xi_{k,i}'}
 } \left(\boldsymbol B^{\xi_k}-\boldsymbol B^{\xi_{k,i}'}\right)\cdot\nabla_{X^N} \frac{
 f_t^{\xi_k}f_t^{\xi_{k,i}'}
 }{
 f_t^{\xi_k}+f_t^{\xi_{k,i}'}
 } \,\mathrm dX^N.
\end{align*}
Dropping the non-positive diffusion term in \eqref{eq:rbm-D-exact} and using the 
Cauchy--Schwarz inequality  with weight $f_t^{\xi_k}+f_t^{\xi_{k,i}'}$ gives
\begin{align}
\label{eq:rbm-D-Cauchy}
 \frac{\mathrm d}{\mathrm dt}\mathscr D_i
 \le
 4\mathscr D_i^{1/2}
 \left[
 \int_{\mathbb R^{2N}}
 \frac{1}{f_t^{\xi_k}+f_t^{\xi_{k,i}'}}
 \left| \left(\boldsymbol B^{\xi_k}-\boldsymbol B^{\xi_{k,i}'}\right)\cdot\nabla_{X^N} \frac{
 f_t^{\xi_k}f_t^{\xi_{k,i}'}
 }{
 f_t^{\xi_k}+f_t^{\xi_{k,i}'}
 }\right|^2
 \,\mathrm dX^N
 \right]^{1/2}.
\end{align}
Note that the vector \(\boldsymbol B^{\xi_k}-\boldsymbol B^{\xi_{k,i}'}\) is supported on \(\Lambda_i\) by \eqref{eq:rbm-drift-localization}-\eqref{eq:rbm-full-drift-difference}. 
Moreover, 
\[
 \nabla_{\Lambda_i}\frac{
 f_t^{\xi_k}f_t^{\xi_{k,i}'}
 }{
 f_t^{\xi_k}+f_t^{\xi_{k,i}'}
 }
 =
 \frac{\left(f_t^{\xi_{k,i}'}\right)^2\nabla_{\Lambda_i} f_t^{\xi_k}
 +
  \left(f_t^{\xi_k}\right)^2\nabla_{\Lambda_i}f_t^{\xi_{k,i}'}
 }{
\left( f_t^{\xi_k}+f_t^{\xi_{k,i}'}\right)^2
 }.
\]
Using $
 |a+b|^2\le2|a|^2+2|b|^2
$ and $ \frac{a^2b^4}{(a+b)^5}
 \le
 a$, one can obtain the pointwise bound
\begin{equation}
\label{eq:rbm-overlap-gradient}
 \frac{
 \left|\nabla_{\Lambda_i}\frac{
 f_t^{\xi_k}f_t^{\xi_{k,i}'}
 }{
 f_t^{\xi_k}+f_t^{\xi_{k,i}'}
 }\right|^2
 }{
f_t^{\xi_k}+f_t^{\xi_{k,i}'}
 }
 \le
 2f_t^{\xi_k}
 \left|\nabla_{\Lambda_i}\log f_t^{\xi_k}\right|^2
 +
 2f_t^{\xi_{k,i}'}
 \left|\nabla_{\Lambda_i}\log f_t^{\xi_{k,i}'}\right|^2.
\end{equation}
Combining
\eqref{eq:rbm-full-drift-difference} and
\eqref{eq:rbm-overlap-gradient} yields
\begin{align}
\label{eq:rbm-source-bound}
 & \int_{\mathbb R^{2N}}
 \frac{1}{f_t^{\xi_k}+f_t^{\xi_{k,i}'}}
 \left| \left(\boldsymbol B^{\xi_k}-\boldsymbol B^{\xi_{k,i}'}\right)\cdot\nabla_{X^N} \frac{
 f_t^{\xi_k}f_t^{\xi_{k,i}'}
 }{
 f_t^{\xi_k}+f_t^{\xi_{k,i}'}
 }\right|^2
 \,\mathrm dX^N\le
 16pM_\varepsilon^2
 \left[
 \mathcal J_{\Lambda_i}(f_t^{\xi_k})
 +
 \mathcal J_{\Lambda_i}(f_t^{\xi_{k,i}'})
 \right].
\end{align}
Thus
\begin{align}
\label{eq:rbm-D-square-root}
 \frac{\mathrm d}{\mathrm dt}\mathscr D_i
 \le
 16\sqrt p\,M_\varepsilon
 \mathscr D_i^{1/2}
 \left[
 \mathcal J_{\Lambda_i}(f_t^{\xi_k})
 +
 \mathcal J_{\Lambda_i}(f_t^{\xi_{k,i}'})
 \right]^{1/2}.
\end{align}
Since both fixed-partition laws start from the common density \(f_{t_k}\), we have
$
  \mathscr D_i(t_k;\xi_k,\xi_{k,i}')=0.
$
Applying the preceding differential inequality to
\((\mathscr D_i+\delta)^{1/2}\), integrating over \([t_k,t]\), and
then letting \(\delta\downarrow0\), we obtain
\[
  \mathscr D_i(t;\xi_k,\xi_{k,i}')^{1/2}
  \le
  8\sqrt p\,M_\varepsilon
  \int_{t_k}^t
  \left[
    \mathcal J_{\Lambda_i}(f_s^{\xi_k})
    +
    \mathcal J_{\Lambda_i}(f_s^{\xi_{k,i}'})
  \right]^{1/2}
  \,\mathrm ds.
\]
Squaring this estimate and applying the Cauchy--Schwarz inequality in
time yield
\begin{align}
\label{eq:rbm-D-local-Fisher}
 \mathscr D_i(t;\xi_k,\xi_{k,i}')
 \le
 64pM_\varepsilon^2(t-t_k)
 \int_{t_k}^t
 \left[
 \mathcal J_{\Lambda_i}(f_s^{\xi_k})
 +
 \mathcal J_{\Lambda_i}(f_s^{\xi_{k,i}'})
 \right]
 \,\mathrm ds.
\end{align}

\medskip
\noindent
\textbf{Step 3: averaging the local Fisher information.}\\
Set
\begin{equation}
\label{eq:rbm-Jbar}
 \overline{\mathcal J}_{N,k}(s)
 :=
 \mathbb E_{\xi_k}
 \mathcal J_N(f_s^{\xi_k})= \mathbb E_{\xi_{k,i}'}
 \mathcal J_N(f_s^{\xi_{k,i}'}).
\end{equation}
Applying \eqref{eq:rbm-combinatorial-count} for a fixed partition with
\[
 A_a
 :=
 \int_{\mathbb R^{2N}}
 f_s^{\xi_k}
 |\nabla_{x_a}\log f_s^{\xi_k}|^2
 \,\mathrm dX^N,
\]
and then averaging over \(\xi_k\), we have
\begin{equation}
\label{eq:rbm-first-local-Fisher-average}
 \frac1N
 \sum_{i=1}^N
 \mathbb E_i
 \mathcal J_{\Lambda_i}(f_s^{\xi_k})
 \le
 2p\,\overline{\mathcal J}_{N,k}(s).
\end{equation}
Similarly,
\begin{equation}
\label{eq:rbm-second-local-Fisher-average}
 \frac1N
 \sum_{i=1}^N
 \mathbb E_i
 \mathcal J_{\Lambda_i}(f_s^{\xi_{k,i}'})
 \le
 2p\,\overline{\mathcal J}_{N,k}(s).
\end{equation}
By averaging \eqref{eq:rbm-D-local-Fisher} and using
\eqref{eq:rbm-first-local-Fisher-average}--
\eqref{eq:rbm-second-local-Fisher-average}, we have
\begin{equation}
\label{eq:rbm-D-average}
 \frac1N
 \sum_{i=1}^N
 \mathbb E_i
 \mathscr D_i(t;\xi_k,\xi_{k,i}')
 \le
 256p^2M_\varepsilon^2(t-t_k)
 \int_{t_k}^t
 \overline{\mathcal J}_{N,k}(s)\,\mathrm ds.
\end{equation}
Combining \eqref{eq:rbm-by-D} and \eqref{eq:rbm-D-average} gives
\begin{equation}
\label{eq:rbm-pointwise-bound}
 |\mathcal R_{\mathrm{RBM}}(t)|
 \le
 \frac\sigma8\mathcal I_N(t)
 +
 \frac{1024p^2}{\sigma}
 M_\varepsilon^4(t-t_k)
 \int_{t_k}^t
 \overline{\mathcal J}_{N,k}(s)\,\mathrm ds.
\end{equation}

\medskip
\noindent
\textbf{Step 4: entropy dissipation and summation over the time steps.}\\
Let
\[
 T_k:=t_{k+1}\wedge T.
\]
By Fubini's theorem, one has
\begin{align}
\label{eq:rbm-one-step-bound}
 \int_{t_k}^{T_k}
 |\mathcal R_{\mathrm{RBM}}(t)|\,\mathrm dt
 \le{}&
 \frac\sigma8
 \int_{t_k}^{T_k}
 \mathcal I_N(t)\,\mathrm dt+
 \frac{512p^2}{\sigma}
 M_\varepsilon^4\tau^2
 \int_{t_k}^{T_k}
 \overline{\mathcal J}_{N,k}(s)\,\mathrm ds.
\end{align}
For each fixed partition, multiplying
\eqref{eq:rbm-fixed-partition-fp} by
$
 \frac1N(1+\log f_t^{\xi_k})$,
and then integrating by parts and using the divergence-free condition for $\boldsymbol{B^{\xi_k}}$ 
\eqref{eq:rbm-full-drift-div-free}, we have
\begin{equation}
\label{eq:rbm-Boltzmann-dissipation}
 \frac{\mathrm d}{\mathrm dt}
 \mathcal S_N(f_t^{\xi_k})
 =
 -\sigma\mathcal J_N(f_t^{\xi_k}).
\end{equation}
Since
\[
 f_{t_k}^{\xi_k}=f_{t_k},
 \qquad
 f_{T_k}
 =
 \mathbb E_{\xi_k}f_{T_k}^{\xi_k},
\]
we have
\begin{align*}
 \sigma
 \int_{t_k}^{T_k}
 \overline{\mathcal J}_{N,k}(s)\,\mathrm ds
=
 \mathcal S_N(f_{t_k})
 -
 \mathbb E_{\xi_k}
 \mathcal S_N(f_{T_k}^{\xi_k})\le
 \mathcal S_N(f_{t_k})
 -
 \mathcal S_N(f_{T_k}),
\end{align*}
where the last inequality follows from the convexity of
$
 F\longmapsto\int F\log F.
$
Now by combining over all steps intersecting \([0,T]\), we obtain
\begin{equation}
\label{eq:rbm-global-Fisher-telescope}
 \sigma
 \sum_{k:t_k<T}
 \int_{t_k}^{t_{k+1}\wedge T}
 \overline{\mathcal J}_{N,k}(s)\,\mathrm ds
 \le
 \mathcal S_N(f_{0})-\mathcal S_N(f_{T}).
\end{equation}

\medskip
\noindent
\textbf{Step 5: lower bound for the terminal entropy.}
\\For each fixed partition, the Liouville equation \eqref{eq:rbm-fixed-partition-fp} gives the following evolution equation for the second moment
\begin{equation}
\label{eq:rbm-moment-evolution}
 \frac{\mathrm d}{\mathrm dt}
 \mathcal M_{2,N}(f_t^{\xi_k})
 =
 \frac2N
 \int_{\mathbb R^{2N}}
 \sum_{a=1}^N
 x_a\cdot B_a^{\xi_k}
 f_t^{\xi_k}\,\mathrm dX^N
 +
 4\sigma.
\end{equation}
By using the oddness of \(K_\varepsilon\) and pairing the two orientations of
each unordered pair in a batch, one has
\begin{equation}
\label{eq:rbm-pairing-identity}
 \sum_{a=1}^N
 x_a\cdot B_a^{\xi_k}
 =
 \frac1{p-1}
 \sum_{C\in\xi_k}
 \sum_{\substack{a,b\in C\\a<b}}
 (x_a-x_b)\cdot K_\varepsilon(x_a-x_b).
\end{equation}
Here $C$ ranges over the batches in the partition $\xi_k$. Since there are \(N(p-1)/2\) unordered within-batch pairs and \(|z|\,|K_\eps(z)|\le A_\varphi\), one has
\[
 \left|
 \frac2N
 \sum_{a=1}^N
 x_a\cdot B_a^{\xi_k}
 \right|
 \le
 A_\varphi.
\] 
Then averaging and iterating over the time steps give
\begin{equation}
\label{eq:rbm-moment-bound}
 \mathcal M_{2,N}(f_t)
 \le
 \mathcal M_{2,N}(f_0)
 +(A_\varphi+4\sigma)t,
 \qquad
 0\le t\le T.
\end{equation}
Note that for any probability density $F$, one has
\begin{equation}
\label{eq:rbm-entropy-lower-bound}
 \mathcal S_N(F)
 \ge
 -\log\pi-\mathcal M_{2,N}(F).
\end{equation}
Combining \eqref{eq:rbm-global-Fisher-telescope},
\eqref{eq:rbm-moment-bound}, and
\eqref{eq:rbm-entropy-lower-bound} yields
\begin{align}
\label{eq:rbm-global-absolute-Fisher}
 &\sum_{k:t_k<T}
 \int_{t_k}^{t_{k+1}\wedge T}
 \overline{\mathcal J}_{N,k}(s)\,\mathrm ds
 \le
 \frac1\sigma
 \Bigl[
 \mathcal S_N(f_0)
 +\log\pi
 +\mathcal M_{2,N}(f_0)
 +(A_\varphi+4\sigma)T
 \Bigr].
\end{align}
Summing \eqref{eq:rbm-one-step-bound} over \(k\) and using
\eqref{eq:rbm-global-absolute-Fisher} proves
\eqref{eq:rbm-integrated-general}.

It remains to verify the two blob-kernel bounds. By the homogeneity
$
 K(\varepsilon z)=\varepsilon^{-1}K(z)
$, one has
$
 K_\varepsilon(\varepsilon z)
 =
 \varepsilon^{-1}(K*\varphi)(z),
$
which proves \eqref{eq:rbm-kernel-scaling}. With
$
 w=z/\varepsilon,
$
one has
\[
 |z|\,|K_\varepsilon(z)|
 =
 |w|\,|(K*\varphi)(w)|.
\]
The function \(K*\varphi\) is bounded on bounded sets. If
$
 \operatorname{supp}\varphi\subset B_R
$
and \(|w|\ge2R\), then
\[
 |(K*\varphi)(w)|
 \le
 \int_{B_R}
 \frac{C\varphi(y)}{|w-y|}\,\mathrm dy
 \le
 \frac C{|w|}.
\]
Hence \[
 A_\varphi
 :=
 \sup_{0<\varepsilon\le1}
 \sup_{z\in\mathbb R^2}
 |z|\,|K_\varepsilon(z)|
 <\infty.
\]
Finally, choose \(0<\lambda<b_0\). The entropy variational inequality gives
\begin{equation}
\label{eq:rbm-initial-moment-from-entropy}
 \lambda\mathcal M_{2,N}(f_0)
 \le
\mathcal H_N\left(f_0 \mid \omega_0^{\otimes N}\right)
 +
 \log
 \int_{\mathbb R^2}
 e^{\lambda|x|^2}\omega_0(x)\,\mathrm dx.
\end{equation}
By the lower bound on \(\log\omega_0\) together with
\eqref{eq:rbm-initial-moment-from-entropy}, one has
\[
 \mathcal S_N(f_0)
 =
 \mathcal H_N\left(f_0 \mid \omega_0^{\otimes N}\right)
 +
 \frac1N
 \int_{\mathbb R^{2N}}
 f_0\log\omega_0^{\otimes N}\,\mathrm dX^N.
\]
Since \(\omega_0\in L^\infty\),
\[
 \mathcal S_N(f_0)
 \le
 \mathcal H_N\left(f_0 \mid \omega_0^{\otimes N}\right)+\log\|\omega_0\|_{L^\infty}.
\]
Thus the square bracket in \eqref{eq:rbm-integrated-general} is bounded by
$
 C_{T,\omega_0}\left(1+\mathcal H_N\left(f_0 \mid \omega_0^{\otimes N}\right)\right),
$
which proves \eqref{eq:rbm-integrated-relative-initial}.
\end{proof}
\section{The mean field contribution}
\label{sec:mean-field-contribution}

In this section, we estimate the mean field contribution
\(\mathcal R_{\mathrm{mf}}\) \eqref{eq:R-mf-score} in the relative entropy identity. The main
difficulty is to obtain a bound that is uniform in the blob radius
\(\varepsilon\). A direct
bounded-kernel implementation of the framework of \cite[Lemma~4.5]{HuangJinLi2025} is not uniform with respect to
the blob radius. Indeed, the crude estimate
$
  \|K_\varepsilon\|_{L^\infty}
  \lesssim \varepsilon^{-1}
$ and Grönwall's inequality
would lead to 
\[
  \mathcal H_N(t)
  \lesssim
  \exp\!\left(\frac{C_T}{\varepsilon^2}\right)
  \left[
    \mathcal H_N(0)
    +
    \text{consistency errors}
  \right].
\]
Such an exponentially singular dependence is unsuitable for the
vanishing-blob limit and is also computationally restrictive. The proof argument in this section consists of two steps which are similar to \cite{FengWang2026}. We first establish Gaussian bounds and
logarithmic derivative estimates for the regularized reference solution
\(\omega_\varepsilon\) with constants uniform for
\(0<\varepsilon\leq1\). These estimates provide the spatial decay and the
control of the score and logarithmic Hessian needed in the fluctuation
analysis. We then exploit the oddness and divergence-free structure of the
regularized Biot--Savart kernel to symmetrize the particle fluctuation. The
resulting kernel contains the score difference
\[
  \nabla\log\omega_\varepsilon(x)
  -
  \nabla\log\omega_\varepsilon(y),
\]
which cancels the near-diagonal singularity of \(K_\varepsilon(x-y)\).
Combining this cancellation with the canonical exponential estimate of
Jabin-Wang \cite{JabinWang2018} yields the desired bound for $\cR_{mf}$.
\subsection{Uniform Gaussian and logarithmic estimates}
We first establish the following Gaussian and Hamilton-type logarithmic gradient estimate for $\omega_{\varepsilon}$.
\begin{proposition}[Uniform Gaussian and logarithmic estimates]
\label{prop:uniform-reference-estimates}
For $\sigma>0$ and $0<\varepsilon\leq1$, let $\omega_\varepsilon$ be the positive
classical solution of
\begin{equation}
\label{eq:regularized-reference}
\left\{
\begin{aligned}
  \partial_t\omega_\varepsilon
  +
  u_\varepsilon\cdot\nabla\omega_\varepsilon
  &=
  \sigma\Delta\omega_\varepsilon,
  \\
  u_\varepsilon
  &=
  K_\varepsilon*\omega_\varepsilon
  =
  K*(\varphi_\varepsilon*\omega_\varepsilon),
  \\
  \omega_\varepsilon|_{t=0}
  &=
  \omega_0.
\end{aligned}
\right.
\end{equation}
Assume that $\omega_0$ is a strictly positive probability density and that
there exist constants $a_0,C_0,C_1,C_2>0$ such that
\begin{equation}
\label{eq:initial-reference-assumptions}
\begin{aligned}
  \omega_0(x)
  &\leq
  C_0e^{-a_0|x|^2},
  \\
  |\nabla\log\omega_0(x)|^2
  &\leq
  C_1(1+|x|^2),
  \\
  |\nabla^2\log\omega_0(x)|
  &\leq
  C_2(1+|x|^2).
\end{aligned}
\end{equation}
Then, for every $T>0$, there exist constants
$
  a_T,b_T,c_T,C_T>0,
$
depending only on $T$, $\sigma$, and the constants in
\eqref{eq:initial-reference-assumptions}, but independent of
$0<\varepsilon\leq1$, such that
\begin{equation}
\label{eq:two-sided-gaussian-reference}
  c_Te^{-b_T|x|^2}
  \leq
  \omega_\varepsilon(t,x)
  \leq
  C_Te^{-a_T|x|^2},
  \qquad
  0\leq t\leq T.
\end{equation}
Moreover,
\begin{equation}
\label{eq:uniform-score-reference}
  |\nabla\log\omega_\varepsilon(t,x)|
  \leq
  C_T(1+|x|)
\end{equation}
and
\begin{equation}
\label{eq:uniform-log-hessian-reference}
  |\nabla^2\log\omega_\varepsilon(t,x)|
  \leq
  C_T(1+|x|^2).
\end{equation}
\end{proposition}
\begin{proof}
We follow the regularity argument of Ben--Artzi
\cite{ben1994global} and the Hamilton-type estimates of Feng and Wang
\cite{FengWang2026}. The only point that requires verification is that all
constants remain uniform for \(0<\varepsilon\leq1\).

First, the assumptions on \(\omega_0\) imply
$
  \omega_0
  \in
  W^{2,1}(\mathbb R^2)
  \cap
  W^{2,\infty}(\mathbb R^2).
$
Indeed,
\[
  \nabla\omega_0
  =
  \omega_0\nabla\log\omega_0
\]
and
\[
  \nabla^2\omega_0
  =
  \omega_0
  \left(
    \nabla^2\log\omega_0
    +
    \nabla\log\omega_0
    \otimes
    \nabla\log\omega_0
  \right),
\]
so the Gaussian upper bound and
\eqref{eq:initial-reference-assumptions} control the first two derivatives
in \(L^1\cap L^\infty\).

Since \(u_\varepsilon\) is divergence-free, mass is preserved and the
maximum principle gives
\[
  \|\omega_\varepsilon(t)\|_{L^1}=1,
  \qquad
  \|\omega_\varepsilon(t)\|_{L^\infty}
  \leq
  \|\omega_0\|_{L^\infty}.
\]
Moreover, convolution with \(\varphi_\varepsilon\) is contractive in
\(L^1\) and \(L^\infty\), and, in the sense of distributions,
\[
  \nabla^m u_\varepsilon
  =
  K*
  \left(
    \varphi_\varepsilon*
    \nabla^m\omega_\varepsilon
  \right),
  \qquad
  m=0,1,2.
\]
Using the standard Biot--Savart estimate
\[
  \|K*f\|_{L^\infty}
  \leq
  C
  \left(
    \|f\|_{L^1}
    +
    \|f\|_{L^\infty}
  \right),
\]
the finite-time Sobolev argument in
\cite[Lemma~2.2]{FengWang2026}, which is based on the classical
regularity theory of Ben--Artzi \cite{ben1994global}, applies with
\(K*\nabla^m\omega\) replaced by
$
  K*
  \left(
    \varphi_\varepsilon*
    \nabla^m\omega_\varepsilon
  \right).
$
The \(L^1\)- and \(L^\infty\)-contractivity of the mollifier therefore
gives
\begin{equation}
\label{eq:uniform-reference-sobolev-summary}
  \sup_{0<\varepsilon\leq1}
  \sup_{0\leq t\leq T}
  \left[
    \|\omega_\varepsilon(t)\|_{W^{2,1}\cap W^{2,\infty}}
    +
    \|u_\varepsilon(t)\|_{W^{2,\infty}}
  \right]
  \leq
  C_T,
\end{equation}
where $C_T$ is independent of $\varepsilon$.
We next prove the Gaussian bounds. Set
\[
  U_T
  :=
  \sup_{0<\varepsilon\leq1}
  \sup_{0\leq t\leq T}
  \|u_\varepsilon(t)\|_{L^\infty}.
\]
The score bound at \(t=0\), together with the normalization of
\(\omega_0\), implies that there exist \(b_0,c_0>0\), depending only on
the initial constants, such that
$
  \omega_0(x)
  \geq
  c_0e^{-b_0|x|^2}.
$
Indeed, the Gaussian upper bound places a fixed positive amount of mass in
a bounded ball, and the lower estimate follows by integrating
\(\nabla\log\omega_0\) along line segments.

Define $
  a(t)
  :=
  \frac{a_0}{1+8\sigma a_0t}
$
and
\[
  \overline\omega(t,x)
  :=
  C_0
  \exp\left(\frac{U_T^2}{4\sigma}t\right)
  e^{-a(t)|x|^2}.
\]
A direct calculation shows that
\[
  \left(
    \partial_t
    +
    u_\varepsilon\cdot\nabla
    -
    \sigma\Delta
  \right)
  \overline\omega
  \geq0.
\]
Similarly, if
\[
  \underline\omega(t,x)
  :=
  c_0
  \exp\left[
    -
    \left(
      4\sigma b_0
      +
      \frac{U_T^2}{4\sigma}
    \right)t
  \right]
  e^{-b_0|x|^2},
\]
then
\[
  \left(
    \partial_t
    +
    u_\varepsilon\cdot\nabla
    -
    \sigma\Delta
  \right)
  \underline\omega
  \leq0.
\]
The whole-space comparison principle therefore yields
\[
  \underline\omega(t,x)
  \leq
  \omega_\varepsilon(t,x)
  \leq
  \overline\omega(t,x),
  \qquad
  0\leq t\leq T.
\]
This proves \eqref{eq:two-sided-gaussian-reference}, with constants
independent of \(\varepsilon\). In particular,
\begin{equation}
\label{eq:uniform-log-density-growth}
  |\log\omega_\varepsilon(t,x)|
  \leq
  C_T(1+|x|^2).
\end{equation}

It remains to prove the logarithmic derivative estimates. Introduce
\[
  \mathscr L_\varepsilon
  :=
  \partial_t
  +
  u_\varepsilon\cdot\nabla
  -
  \sigma\Delta.
\]
Let
\[
  Q_1
  :=
  \frac{|\nabla\omega_\varepsilon|^2}{\omega_\varepsilon}.
\]
The quotient calculation of
\cite[Lemma~4.1]{FengWang2026}, with the diffusion coefficient
\(\sigma\) restored, gives
\[
  \mathscr L_\varepsilon Q_1
  \leq
  C_TQ_1,
  \qquad
  \mathscr L_\varepsilon(\omega_\varepsilon\log\omega_\varepsilon)
  =
  -\sigma Q_1,
\]
where \(C_T\) is independent of \(\varepsilon\) by
\eqref{eq:uniform-reference-sobolev-summary}. Choose \(B_1,B_2>0\), independent of \(\varepsilon\), such that
$
  \sigma B_1\geq C_T
$
and
$
  |\nabla\log\omega_0|^2
  +
  B_1\log\omega_0
  \leq
  B_2.
$
The latter is possible by the initial score bound and the Gaussian upper
bound. The auxiliary function
\[
  F_1
  :=
  Q_1
  +
  B_1\omega_\varepsilon\log\omega_\varepsilon
  -
  B_2\omega_\varepsilon
\]
then satisfies
\[
  \mathscr L_\varepsilon F_1\leq0,
  \qquad
  F_1(0,\cdot)\leq0.
\]
The Sobolev estimates and the Gaussian lower bound show that
\((F_1)_+\) has at most Gaussian-exponential growth. The standard
whole-space maximum principle therefore gives \(F_1\leq0\). Dividing by
\(\omega_\varepsilon\) and using \eqref{eq:uniform-log-density-growth}, we obtain
\[
  |\nabla\log\omega_\varepsilon(t,x)|^2
  \leq
  C_T(1+|x|^2),
\]
which proves \eqref{eq:uniform-score-reference}.

Finally, set
\[
  Q_2
  :=
  \frac{|\nabla^2\omega_\varepsilon|^2}{\omega_\varepsilon}.
\]
The local quotient calculation in
\cite[Lemma~4.2]{FengWang2026} gives, on the finite interval
\([0,T]\),
\[
  \mathscr L_\varepsilon Q_2
  \leq
  C_TQ_2
  +
  C_TQ_1,
\]
and
\[
  \mathscr L_\varepsilon
  \left[
    \omega_\varepsilon(\log\omega_\varepsilon)^2
  \right]
  =
  -2\sigma(1+\log\omega_\varepsilon)Q_1.
\]
The constants are again uniform in \(\varepsilon\), because the calculation
only involves the bounds on \(\nabla u_\varepsilon\) and
\(\nabla^2u_\varepsilon\) in
\eqref{eq:uniform-reference-sobolev-summary}.

Using the same Hamilton auxiliary function as in
\cite[Theorem~4.4]{FengWang2026}, namely
\[
  F_2
  :=
  e^{-\lambda_Tt}Q_2
  -
  B_5\omega_\varepsilon(\log\omega_\varepsilon)^2
  +
  B_6\omega_\varepsilon\log\omega_\varepsilon
  -
  B_7\omega_\varepsilon,
\]
one may choose \(\lambda_T,B_5,B_6,B_7>0\), independently of
\(\varepsilon\), so that
\[
  \mathscr L_\varepsilon F_2\leq0,
  \qquad
  F_2(0,\cdot)\leq0.
\]
Here the initial inequality follows from
\[
  \frac{\nabla^2\omega_0}{\omega_0}
  =
  \nabla^2\log\omega_0
  +
  \nabla\log\omega_0
  \otimes
  \nabla\log\omega_0
\]
and the Gaussian upper bound. The same whole-space maximum principle gives
\(F_2\leq0\), and hence
\[
  \frac{|\nabla^2\omega_\varepsilon(t,x)|^2}
  {\omega_\varepsilon(t,x)^2}
  \leq
  C_T
  \left(
    1+
    |\log\omega_\varepsilon(t,x)|^2
  \right)
  \leq
  C_T(1+|x|^4).
\]
Therefore,
\[
  \frac{|\nabla^2\omega_\varepsilon(t,x)|}
  {\omega_\varepsilon(t,x)}
  \leq
  C_T(1+|x|^2).
\]
Since
\[
  \nabla^2\log\omega_\varepsilon
  =
  \frac{\nabla^2\omega_\varepsilon}{\omega_\varepsilon}
  -
  \nabla\log\omega_\varepsilon
  \otimes
  \nabla\log\omega_\varepsilon,
\]
the logarithmic gradient estimate proves
\[
  |\nabla^2\log\omega_\varepsilon(t,x)|
  \leq
  C_T(1+|x|^2).
\]
This is \eqref{eq:uniform-log-hessian-reference}.
\end{proof}

\subsection{Mean field Biot--Savart fluctuation on $\mathbb R^2$}
Once again, we recall the large deviation type estimate \cite[Theorem~4]{JabinWang2018}.
\begin{lemma}[Jabin-Wang]\label{theorem:JW}
 Consider any $\phi(x, y)$ satisfying the canceling properties
\[
\int_{\mathbb{R}^2} \phi(x, y) \bar{\rho}(x) \mathrm{d} x=0 \text { for any } y, \quad \int_{\mathbb{R}^2} \phi(x, y) \bar{\rho}(y) \mathrm{d} y=0 \text { for any } x,
\]
and there exists a universal constant $C_{J W}=1600^2+36 e^4$ such that
\[
\gamma=C_{J W}\left(\sup _{q \geq 1} \frac{\left\|\sup _y|\phi(\cdot, y)|\right\|_{L^q}(\bar{\rho} \mathrm{~d} x)}{q}\right)^2<1 .
\]
Then we have
\[
\int_{\mathbb{R}^{2 N}} \bar\rho^{\otimes N} \exp \left(\frac{1}{N} \sum_{i, j=1}^N \phi\left(x_i, x_j\right)\right) \mathrm{d} X^N \leq \frac{2}{1-\gamma}<\infty .
\] 
\end{lemma} 

\begin{lemma}[mean field Biot--Savart fluctuation on $\mathbb R^2$]
\label{lem:mean-field-whole-space}
Assume that $\varphi$ is even and that, uniformly for
$0<\varepsilon\leq1$ and $0\leq t\leq T$,
\begin{equation}
\label{eq:blob-structural-bounds}
  K_\varepsilon(-z)
  =
  -K_\varepsilon(z),
  \qquad
  \operatorname{div}K_\varepsilon
  =
  0,
  \qquad
  |K_\varepsilon(z)|
  \leq
  \frac{C}{|z|+\varepsilon},
  \qquad
  z\in\mathbb R^2.
\end{equation}
Assume also that $\omega_\varepsilon(t,\cdot)$ is a strictly positive
probability density and that there exist constants $a_T,C_T>0$, independent
of $\varepsilon$, such that
\begin{equation}
\label{eq:whole-space-reference-estimates}
\begin{aligned}
  \omega_\varepsilon(t,x)
  &\leq
  C_Te^{-a_T|x|^2},
  \\
  |\nabla\log\omega_\varepsilon(t,x)|
  &\leq
  C_T(1+|x|),
  \\
  |\nabla^2\log\omega_\varepsilon(t,x)|
  &\leq
  C_T(1+|x|^2)
\end{aligned}
\qquad
x\in\mathbb R^2.
\end{equation}
Then there exists a constant, still denoted by $C_T$, independent of
$N,\tau,$ and $\varepsilon$, such that
\begin{equation}
\label{eq:mean-field-whole-space-estimate}
  \mathcal R_{\mathrm{mf}}(t)
  \leq
  C_T
  \left[
  \mathcal H_N
  \bigl(
  \tilde F^N_\varepsilon(t)
  \mid
  \omega_\varepsilon(t)^{\otimes N}
  \bigr)
  +
  \frac1N
  \right],
  \qquad
  0\leq t\leq T.
\end{equation}
\end{lemma}

\begin{proof}
By \eqref{eq:R-mf-score},
\begin{align}
\label{eq:R-mf-before-symmetrization}
\mathcal R_{\mathrm{mf}}(t)
=
-\int_{\mathbb R^{2N}}
\tilde F^N_\varepsilon(t,X^N)
\Bigg\{
&
\frac1{N(N-1)}
\sum_{i\ne j}
K_\varepsilon(x_i-x_j)
\cdot
\nabla\log\omega_\varepsilon(x_i)
\notag\\
&-
\frac1N
\sum_{i=1}^N
(K_\varepsilon*\omega_\varepsilon)(x_i)
\cdot
\nabla\log\omega_\varepsilon(x_i)
\Bigg\}
\,\mathrm dX^N.
\end{align}
We follow a decomposition similar to that in \cite[Section~5]{FengWang2026}. The first term in the integrand of \eqref{eq:R-mf-before-symmetrization} is a two-particle average, while the second term is a one-particle average. One can symmetrize the two-particle average and rewrite the one-particle average as a two-particle average by
using the oddness of $K_\varepsilon$ and exchanging $i$ and $j$,
\begin{align}
\label{eq:odd-symmetrization}
&
\frac1{N(N-1)}
\sum_{i\ne j}
K_\varepsilon(x_i-x_j)
\cdot
\nabla\log\omega_\varepsilon(x_i)
\notag\\
&\qquad
=
\frac1{2N(N-1)}
\sum_{i\ne j}
K_\varepsilon(x_i-x_j)
\cdot
\bigl(
\nabla\log\omega_\varepsilon(x_i)
-
\nabla\log\omega_\varepsilon(x_j)
\bigr).
\end{align}
Moreover,
\begin{align}
\label{eq:one-particle-off-diagonal}
&
\frac1N
\sum_{i=1}^N
(K_\varepsilon*\omega_\varepsilon)(x_i)
\cdot
\nabla\log\omega_\varepsilon(x_i)
\notag\\
&\qquad
=
\frac1{2N(N-1)}
\sum_{i\ne j}
\Bigl[
(K_\varepsilon*\omega_\varepsilon)(x_i)
\cdot
\nabla\log\omega_\varepsilon(x_i)
+
(K_\varepsilon*\omega_\varepsilon)(x_j)
\cdot
\nabla\log\omega_\varepsilon(x_j)
\Bigr].
\end{align}
For $x\ne y$, define
\begin{align}
\label{eq:canonical-kernel-definition}
\Phi_\varepsilon(t,x,y)
:={}&
\frac12
K_\varepsilon(x-y)
\cdot
\bigl[
\nabla\log\omega_\varepsilon(x)
-
\nabla\log\omega_\varepsilon(y)
\bigr]
\notag\\
&-
\frac12
(K_\varepsilon*\omega_\varepsilon)(x)
\cdot
\nabla\log\omega_\varepsilon(x)
-
\frac12
(K_\varepsilon*\omega_\varepsilon)(y)
\cdot
\nabla\log\omega_\varepsilon(y),
\end{align}
and set
\[
  \Phi_\varepsilon(t,x,x)
  :=
  0.
\]
Then
\eqref{eq:R-mf-before-symmetrization}--
\eqref{eq:one-particle-off-diagonal}
give
\begin{equation}
\label{eq:R-mf-canonical-average}
  \mathcal R_{\mathrm{mf}}(t)
  =
  -
  \int_{\mathbb R^{2N}}
  \tilde F^N_\varepsilon(t,X^N)
  \frac1{N(N-1)}
  \sum_{i\ne j}
  \Phi_\varepsilon(t,x_i,x_j)
  \,\mathrm dX^N.
\end{equation}
It is straightforward to verify that $\Phi_\varepsilon(t, x, y)$ satisfies the cancellation condition:
\begin{equation}
\label{eq:canonical-cancellation}
  \int_{\mathbb R^2}
  \Phi_\varepsilon(t,x,y)
  \omega_\varepsilon(t,x)
  \,\mathrm dx
  =
  0
  \quad
  \text{for every }y; \qquad   \int_{\mathbb R^2}
  \Phi_\varepsilon(t,x,y)
  \omega_\varepsilon(t,y)
  \,\mathrm dy
  =
  0
  \quad
  \text{for every }x.
\end{equation}
Now one can establish the growth bound needed for the whole-space canonical
large-deviation estimate. First, as in 
\cite{FengWang2026}, one can claim that
\begin{equation}
\label{eq:velocity-decay}
  |(K_\varepsilon*\omega_\varepsilon)(x)|
  \leq
  \frac{C_T}{1+|x|},
\end{equation} where $C_T$ is independent of $\varepsilon$. Indeed, 
for $|x|\leq4$ using \eqref{eq:blob-structural-bounds},
\eqref{eq:whole-space-reference-estimates}, and
$\int\omega_\varepsilon=1$, one has
\begin{align*}
|(K_\varepsilon*\omega_\varepsilon)(x)|
&\leq
\int_{\mathbb R^2}
\frac{C}{
|x-y|+\varepsilon
}
\omega_\varepsilon(t,y)
\,\mathrm dy
\\
&\leq
C
\|\omega_\varepsilon(t)\|_{L^\infty}
\int_{|z|\leq1}
\frac{\,\mathrm dz}{|z|}
+
C
\int_{|z|>1}
\omega_\varepsilon(t,x-z)
\,\mathrm dz
\\
&\leq
C_T.
\end{align*} 
For $|x|>4$, split the integral defining $(K_\varepsilon*\omega_\varepsilon)(x)$ into
\[
  \{|y|\leq |x|/2\},
  \qquad
  \{|y|>|x|/2,\ |x-y|\leq1\},
  \qquad
  \{|y|>|x|/2,\ |x-y|>1\}.
\]
For the first region $
  |x-y|
  \geq
  \frac{|x|}{2}$,
and hence
\begin{align*}
\int_{\{|y|\leq|x|/2\}}
|K_\varepsilon(x-y)|
\omega_\varepsilon(t,y)
\,\mathrm dy
\leq
\frac{C}{|x|}
\int_{\mathbb R^2}
\omega_\varepsilon(t,y)
\,\mathrm dy
=
\frac{C}{|x|}.
\end{align*}
On the second region, $|x-y|\leq1$ and $|x|>4$ imply $
  |y|
  \geq
  |x|-1
  \geq
  \frac{|x|}{2}$. Consequently,
\begin{align*}
\int_{\substack{|y|>|x|/2\\|x-y|\leq1}}
|K_\varepsilon(x-y)|
\omega_\varepsilon(t,y)
\,\mathrm dy
\leq
C_T
e^{-a_T|x|^2/4}
\int_{|z|\leq1}
\frac{\,\mathrm dz}{|z|}
\leq
C_Te^{-a_T|x|^2/4}.
\end{align*}
Finally, on the third region $|x-y|>1$, and hence
\begin{align*}
\int_{\substack{|y|>|x|/2\\|x-y|>1}}
|K_\varepsilon(x-y)|
\omega_\varepsilon(t,y)
\,\mathrm dy
\leq
C
\int_{|y|>|x|/2}
\omega_\varepsilon(t,y)
\,\mathrm dy
\leq
C_T
\int_{|y|>|x|/2}
e^{-a_T|y|^2}
\,\mathrm dy
\leq
C_Te^{-a_T|x|^2/8}.
\end{align*}
This proves \eqref{eq:velocity-decay}. Combining it with the score estimate
in \eqref{eq:whole-space-reference-estimates}, one obtains
\begin{equation}
\label{eq:velocity-score-bound}
  \sup_{x\in\mathbb R^2}
  |
  (K_\varepsilon*\omega_\varepsilon)(x)
  \cdot
  \nabla\log\omega_\varepsilon(x)
  |
  \leq
  C_T,
\end{equation} 
where $C_T$ is independent of $\varepsilon$.
Next one can estimate the first term in 
\eqref{eq:canonical-kernel-definition}. Let $r
  :=
  |x-y|.$ If $r\leq1$, then by the mean-value formula
\begin{align*}
|K_\varepsilon(x-y)|
\,
|\nabla\log\omega_\varepsilon(x)-\nabla\log\omega_\varepsilon(y)|
\leq
&C
\frac{r}{r+\varepsilon}
\int_0^1
\left|
\nabla^2
\log
\omega_\varepsilon
\bigl(
t,y+\theta(x-y)
\bigr)
\right|
\,\mathrm d\theta
\\
\leq &
  C_T(1+|x|^2).
\end{align*}
If $r>1$, then $|K_\varepsilon(x-y)|
  \leq
  \frac{C}{r} $.
Thus
\begin{align}
\label{eq:far-score-difference}
|K_\varepsilon(x-y)|
\,
|\nabla\log\omega_\varepsilon(x)-\nabla\log\omega_\varepsilon(y)|
\leq
\frac{C_T}{r}
\bigl(
2+|x|+|y|
\bigr)
\leq
C_T(1+|x|).
\end{align}
Combining
\eqref{eq:velocity-score-bound} and
\eqref{eq:far-score-difference},
\begin{equation}
\label{eq:canonical-kernel-growth}
  \sup_{y\in\mathbb R^2}
  |
  \Phi_\varepsilon(t,x,y)
  |
  \leq
  C_T(1+|x|^2),
\end{equation} where $C_T$ is independent of $\varepsilon$.
The Gaussian upper bound now yields the required Orlicz estimate. Indeed,
\[
  \int_{\mathbb R^2}
  |x|^{2q}\omega_\varepsilon(t,x)
  \,\mathrm dx
  \leq
  C_T
  \int_0^\infty
  r^{2q+1}e^{-a_Tr^2}
  \,\mathrm dr
  =
  C_T'a_T^{-(q+1)}\Gamma(q+1),
\]
and hence
\begin{equation}
\label{eq:canonical-Orlicz-bound}
  \Lambda_T
  :=
  \sup_{q\geq1}
  \frac1q
  \left\|
    \sup_y|\Phi_\varepsilon(t,\cdot,y)|
  \right\|_{L^q(\omega_\varepsilon(t,x)\,\mathrm dx)}
  \leq
  C_T.
\end{equation}
We finally apply the exponential large deviation type estimate. Let $C_{\mathrm{JW}}$ be
the universal constant in Lemma \ref{theorem:JW}.
Choose $\eta_T>0$ which is independent of $N$ and $\varepsilon$ such that
$
  C_{\mathrm{JW}}(2\eta_T\Lambda_T)^2
  \leq
  \frac12.
$
Lemma \ref{theorem:JW} then yields
\begin{equation}
\label{eq:canonical-exponential-bound-applied}
  \int_{\mathbb R^{2N}}
  \exp\left(
    \frac{\eta_T}{N-1}
    \sum_{i\neq j}
    \Phi_\varepsilon(t,x_i,x_j)
  \right)
  \omega_\varepsilon^{\otimes N}
  \,\mathrm dX^N
  \leq
  4.
\end{equation}
Set 
\[
\Psi_{N,\eps}(t,X^N)=\frac{1}{N(N-1)}\sum_{i\neq j} \Phi_\varepsilon(t,x_i,x_j).
\]
Recall the Donsker–Varadhan inequality as in \cite[Lemma~1]{JabinWang2018} with parameter $\eta_T$, we have
\begin{align*}
\cR_{mf}&\leq
  \frac1{\eta_T}
\cH_N\bigl(
    \tilde F^N_\varepsilon
    \mid
    \omega_\varepsilon^{\otimes N}
  \bigr)+
  \frac1{N\eta_T}
  \log
  \int_{\mathbb R^{2N}}
  \exp(N\eta_T \Psi_{N,\eps})
  \omega_\varepsilon^{\otimes N}
  \,\mathrm dX^N
  \\
  &\leq
  \frac1{\eta_T}
  \mathcal \cH_N\bigl(
    \tilde F^N_\varepsilon
    \mid
    \omega_\varepsilon^{\otimes N}
  \bigr)
  +
  \frac{\log4}{N\eta_T}
  \\
  &\leq
  C_T
  \left[
    \mathcal H_N\bigl(
      \tilde F^N_\varepsilon
      \mid
      \omega_\varepsilon^{\otimes N}
    \bigr)
    +
    \frac1N
  \right],
\end{align*}
which concludes the proof.
\end{proof}

\section{Proof of the main theorem}
\label{sec:proof-main}

We now combine the relative entropy identity \eqref{eq:entropy-evolution-rbm-vortex} with the random batch and mean field
estimates obtained above.

\begin{proof}[Proof of Theorem~\ref{thm:main}]

The assumptions on $\omega_0$ and Proposition~\ref{prop:uniform-reference-estimates}
provide the reference estimates required in
Lemma~\ref{lem:mean-field-whole-space}, with constants uniform in
$0<\varepsilon\leq1$. Moreover, the Gaussian upper bound in
\eqref{eq:main-initial-assumptions} gives an exponential second moment, while
the score bound and strict positivity imply
$-\log\omega_0(x)\leq C(1+|x|^2)$ by integrating
$\nabla\log\omega_0$ along a line segment from a point where
$\omega_0$ is bounded below. Hence the relative-initial-data form of
Proposition~\ref{prop:rbm-integrated-estimate} applies.

The relative entropy identity \eqref{eq:entropy-evolution-rbm-vortex} in
Proposition~\ref{prop:entropy-evolution-rbm-vortex} holds on every interval
$(t_k,t_{k+1})$. Since the particle law is continuous at the switching times,
summing these identities and integrating from $0$ to $t\leq T$ gives
\begin{equation}
\label{eq:main-proof-start}
    \mathcal H_N\!\left(
    \tilde F_\varepsilon^N(t)
    \,\middle|\,
    \omega_\varepsilon(t)^{\otimes N}
  \right)
  +
  \sigma\int_0^t\mathcal I_N(s)\,\mathrm ds
  =
  h_N^0
  +
  \int_0^t\mathcal R_{\mathrm{RBM}}(s)\,\mathrm ds
  +
  \int_0^t\mathcal R_{\mathrm{mf}}(s)\,\mathrm ds.
\end{equation}
By Proposition~\ref{prop:rbm-integrated-estimate}, uniformly for
$t\in[0,T]$,
\begin{equation}
\label{eq:main-proof-rbm}
  \int_0^t
  |\mathcal R_{\mathrm{RBM}}(s)|\,\mathrm ds
  \leq
  \frac{\sigma}{8}
  \int_0^t\mathcal I_N(s)\,\mathrm ds
  +
  C_T\varepsilon^{-4}\tau^2(1+h_N^0).
\end{equation}
On the other hand, Lemma~\ref{lem:mean-field-whole-space} yields
\begin{equation}
\label{eq:main-proof-mf}
  \int_0^t
  |\mathcal R_{\mathrm{mf}}(s)|\,\mathrm ds
  \leq
  C_T\int_0^t h_N(s)\,\mathrm ds
  +
  \frac{C_T}{N}.
\end{equation}
Substituting \eqref{eq:main-proof-rbm} and
\eqref{eq:main-proof-mf} into \eqref{eq:main-proof-start}, we obtain
\begin{align*}
\label{eq:main-proof-gronwall}
  &\mathcal H_N\!\left(
    \tilde F_\varepsilon^N(t)
    \,\middle|\,
    \omega_\varepsilon(t)^{\otimes N}
  \right)
  +
  \frac{7\sigma}{8}
  \int_0^t\mathcal I_N(s)\,\mathrm ds\\
  \notag
  \leq &
  h_N^0
  +
  C_T\varepsilon^{-4}\tau^2(1+h_N^0)
  +
  \frac{C_T}{N}
  +
  C_T\int_0^t   \mathcal H_N\!\left(
    \tilde F_\varepsilon^N(s)
    \,\middle|\,
    \omega_\varepsilon(s)^{\otimes N}
  \right)\,\mathrm ds.
\end{align*}
Gronwall's inequality therefore gives
\[
  \sup_{0\leq t\leq T}  \mathcal H_N\!\left(
    \tilde F_\varepsilon^N(t)
    \,\middle|\,
    \omega_\varepsilon(t)^{\otimes N}
  \right)
  \leq
  C_T\left[
    h_N^0
    +
    \varepsilon^{-4}\tau^2(1+h_N^0)
    +
    \frac1N
  \right].
\]
This proves \eqref{main-result}. If
$\tilde F_\varepsilon^N(0)=\omega_0^{\otimes N}$, then $h_N^0=0$, and
\eqref{main-result-product} follows immediately.
\end{proof}
\section*{Acknowledgements}
The author is grateful to Professor Shi Jin and Professor Lei Li for their valuable guidance and insightful discussions throughout this work. This work was supported by the National Natural Science Foundation of China (Grant No.~12531016), the National Natural Science Foundation of China Tianyuan Fund for Mathematics (Grant No.~12426304), and the Shanghai Municipal Science and Technology Key Project (Grant No.~23JC1402300).
\bibliographystyle{plain}
\bibliography{rbmvortexref}

\end{document}